\documentclass{amsart}

\usepackage{mathtools}
\usepackage{amsfonts}
\usepackage{amssymb}
\usepackage[latin2]{inputenc}
\usepackage{enumerate}
\usepackage{amsmath}
\usepackage{amssymb}
\usepackage{framed,xcolor}
\usepackage{lipsum,hyperref}

\numberwithin{equation}{section}
\newtheorem{theorem}{Theorem}[section]
\newtheorem{lemma}[theorem]{Lemma}
\newtheorem{claim}[theorem]{Claim}

\newtheorem{prop}[theorem]{Proposition}

\newtheorem{question}[theorem]{Question}
\newtheorem{problem}[theorem]{Problem}
\newtheorem{cor}[theorem]{Corollary}

\theoremstyle{definition}             
\newtheorem{remark}[theorem]{Remark}

\newtheorem{example}[theorem]{Example}

\newtheorem{defi}[theorem]{Definition}
\newtheorem{notation}[theorem]{Notation}

\DeclareMathOperator{\dom}{dom}
\DeclareMathOperator{\ubdim}{\overline{dim}_B}
\DeclareMathOperator{\lbdim}{\underline{dim}_B}

\DeclareMathOperator{\hdim}{\dim_H}
\DeclareMathOperator{\PhiMBdim}{\overline{\dim}_{M}^{\Phi}}

\DeclareMathOperator{\PhiBdim}{\overline{\dim}^{\Phi}}

\DeclareMathOperator{\bdim}{\dim_B}
\DeclareMathOperator{\pdim}{\dim_P}
\DeclareMathOperator{\lpdim}{\underline{\dim}_P}
\DeclareMathOperator{\adim}{\dim_A}
\DeclareMathOperator{\ldim}{\dim_L}
\DeclareMathOperator{\mldim}{\dim_{ML}}
\DeclareMathOperator{\drdim}{\dim^D_H}
\DeclareMathOperator{\drtdim}{\dim^{\widetilde{D}}_H}
\DeclareMathOperator{\rdim}{\dim_U}
\DeclareMathOperator{\diam}{diam}

\newcommand{\su}{\subset}

\newcommand{\II}{\mathbb{I}}
\newcommand{\N}{\mathbb{N}}
\newcommand{\Z}{\mathbb{Z}}
\newcommand{\Q}{\mathbb{Q}}
\newcommand{\R}{\mathbb{R}}
\newcommand{\eps}{\varepsilon}

\newcommand{\NN}{\mathbb{N}}
\newcommand{\QQ}{\mathbb{Q}}
\newcommand{\RR}{\mathbb{R}}

\newcommand{\iD}{\mathcal{D}}

\newcommand{\iH}{\mathcal{H}}

\newcommand{\iK}{\mathcal{K}}

\newcommand{\iR}{\mathcal{R}}
\newcommand{\iF}{\mathcal{F}}

\newcommand{\iS}{\mathcal{S}}

\newcommand{\sm}{\setminus}

\title[New Hausdorff-type dimensions]{New Hausdorff-type dimensions and optimal bounds for bilipschitz invariant dimensions}
\author{Rich\'ard Balka} 
\address{HUN-REN Alfr\'ed R\'enyi Institute of Mathematics, Re\'altanoda u.~13--15, H-1053 Budapest, Hungary, AND Institute of Mathematics and Informatics, Eszterh\'azy K\'aroly Catholic University, Le\'anyka u.~4, H-3300 Eger, Hungary}
\email{balkaricsi@gmail.com}

\thanks{The authors were supported by the National Research, Development and Innovation Office -- NKFIH, grants nos.~124749 and 146922. The first author was also supported by the J\'anos Bolyai Research Scholarship of the Hungarian Academy of Sciences.
The second author was also supported by the
National Research, Development and Innovation Office -- NKFIH, grant no.~129335.}
\author{Tam\'as Keleti}
\address{Institute of Mathematics, ELTE E\"otv\"os Lor\'and University, P\'azm\'any P\'eter s\'et\'any 1/c, H-1117 Budapest, Hungary}
\email{tamas.keleti@gmail.com}



\begin{document}

\begin{abstract} 
We introduce a new family of fractal dimensions by restricting the set of diameters in the coverings in the usual definition of the Hausdorff dimension. Among others, we prove that this family contains continuum many distinct dimensions, and they share most of the properties of the Hausdorff dimension, which answers negatively a question of Fraser. On the other hand, we also prove that among these new dimensions only the Hausdorff dimension behaves nicely with respect to H\"older functions. 

We also consider the supremum of these new dimensions, which turns out to be another interesting notion of fractal dimension.

We prove that among those bilipschitz invariant, monotone dimensions on the compact subsets of $\R^n$ that agree with the similarity dimension for the simplest self-similar sets, the modified lower dimension is the smallest and when $n=1$ the Assouad dimension is the greatest, and this latter statement is false for $n>1$. This answers a question of Rutar.
\end{abstract}

\maketitle

\section{Introduction} 
A huge variety of dimensions are widely used in fractal geometry and its applications, the most classical ones are the Hausdorff dimension, the upper and lower box dimensions, and the packing dimension. 
Recently, the Assouad dimension has become popular, and also its dual, the lower dimension, see for example Jonathan Fraser's monograph \cite{F}.
More recently, continuum many new fractal dimensions were added to this list by Falconer, Fraser, and Kempton \cite{FFK}. They introduced the so-called intermediate dimensions, which were further generalized by Amlan Banaji \cite{Ban}.
(See Section~\ref{s:prel} for the definitions and basic properties of the above notions.)

Most of these dimensions share some basic properties: 
\begin{enumerate}
\item\label{i:unity} 
agree with the similarity dimension on every self-similar set satisfying the strong separation condition,
\item \label{i:mon} 
monotonicity: $A\su B\Longrightarrow \dim A\le \dim B$,
\item \label{i:stable} 
finite stability: $\dim(A\cup B)=\max(\dim A, \dim B)$,
\item \label{i:bilip} bilipschitz invariance: $\dim A=\dim B$ if $A$ and $B$ are bilipschitz equivalent.
\end{enumerate}
Some of them also satisfy 
\begin{enumerate}\setcounter{enumi}{4}
\item \label{i:lip}
Lipschitz stability: $\dim f(A)\le \dim A$ if $f\colon \R^n\to \R^n$ is Lipschitz,
\item \label{i:sigma}
$\sigma$-stability: $\dim(\bigcup_{i=1}^\infty A_i)=\sup_i\dim A_i$.
\end{enumerate}

Those that are not $\sigma$-stable can be modified to force this property; in fact, this is a possible way to obtain the packing dimension from the upper box dimension.

Let us study the following inverse problem. 
\begin{problem}\label{p:inverse}
Suppose that we have a notion of dimension that has a given list of properties. What can we say about the dimension?  
\end{problem}

A less ambitious version of this problem is to find optimal lower and upper bounds for dimensions of given properties. The following results of this type was obtained in \cite{BK}: In $\R^n$ any dimension with properties \eqref{i:unity}, \eqref{i:mon} and \eqref{i:lip} is in between the Hausdorff dimension and the upper box dimension for all compact sets, and those with properties \eqref{i:unity}, \eqref{i:mon}, \eqref{i:lip}
and \eqref{i:sigma} are in between the Hausdorff and the packing dimensions. 
Note that the Assouad dimension can be greater than the upper box dimension, but it is not Lipschitz stable, it is only bilipschitz invariant.

Motivated by these, Alex Rutar asked what happens if Lipschitz invariance is replaced by bilipschitz invariance, and suggested that perhaps Assouad dimension $\adim$ is the greatest dimension with properties \eqref{i:unity}, \eqref{i:mon}, and \eqref{i:bilip}.
We settle this question and also find the least dimension with these properties,
which turns out to be the so-called modified lower dimension $\mldim$, see its definition in Section~\ref{s:prel}.

\begin{theorem}\label{t:introdimestimates}
  Let $n$ be a positive integer and let $D$ be a function from the family of compact subsets of $\R^n$ to $[0,n]$. Suppose that 
\begin{itemize}
\item[$(\star)$] \parbox{\linewidth}{$D$ is bilipschitz invariant, monotone, and it agrees with the similarity dimension for any homogeneous self-similar set with the SSC.}
\end{itemize} 
Then for every compact set $K\su \R^n$ we have \begin{equation*} 
\mldim(K)\le D(K).
\end{equation*}
If $n=1$ then for any compact $K$ we also have \begin{equation*} 
D(K)\le\adim K,
\end{equation*} 
but this latter statement is false when $n>1$.
\end{theorem}

Using the notion of Lipschitz decomposability (see Definition~\ref{d:Ldec}) introduced recently by  K\"aenm\"aki and Rutar \cite{KR}
we give a condition, which is stronger than bilipschitz invariance but weaker than Lipschitz stability,
and we prove in Theorem~\ref{t:decomp} that if we replace bilipschitz invariance by this condition in $(\star)$ of the 
above theorem then we obtain Assouad dimension as the greatest dimension
in $\R^n$ for every $n$ with this modified list of properties.


The most ambitious and exciting problem of the type of Problem~\ref{p:inverse} was posed by Jonathan M.~Fraser:
\begin{problem}[Fraser]
Give a list of natural properties of dimensions that uniquely characterize the Hausdorff dimension.
\end{problem}
All the above properties (1)--(6) are satisfied by the Hausdorff dimension, but also by the packing dimension and by the $\sigma$-stable modifications of the (generalized) intermediate dimensions. However, the following property is known to be true only for the Hausdorff dimension, and was proved to be false by Mattila and Mauldin \cite{MM} for the packing dimension:
\begin{enumerate}\setcounter{enumi}{6}
\item \label{i:meas} measurability: $\dim\colon \iK(\R^n)\to [0,n]$ is a Borel function, where $\iK(\R^n)$ is the set of compact subsets of $\R^n$ endowed with the Hausdorff metric. 
\end{enumerate}

This led Fraser to pose the following question.

\begin{question}[Fraser]\label{q:Fraser}
If a function $\dim \colon\iK(\R^n)\to [0,n]$ satisfies all the above properties (1)--(7), then does it imply that this is the Hausdorff dimension?
\end{question}
We give a negative answer to this question by introducing a new family of continuum many distinct dimensions that satisfy all the above properties (1-7).

We obtain these dimensions in the same way as the Hausdorff dimension is defined, but in the covering we allow only sets with diameter from a fixed set $D\su(0,\infty)$, see the details in Definition~\ref{d:restdim}.
To get a reasonable notion, we require $\inf D=0$ to allow sets with arbitrarily small diameter. In addition, to guarantee that any set in the cover can be replaced by a slightly larger open set, we also require that for every $t\in D$, our set $D$ contains an interval of the form $[t,t+\delta)$, in other words, we assume that $D$ is Sorgenfrey open. We denote by $\drdim$ the dimension we obtain from a given such $D$ and we call it \emph{$D$-diameter restricted Hausdorff dimension}.
Note that in the $D=(0,\infty)$ case there is no restriction, so in this case $\drdim$ is the same as the Hausdorff dimension. 

We prove that we indeed obtain continuum many distinct dimensions by proving an interpolation result:
for any compact set defined by digit restrictions (see \cite{BP}) we show that its $D$-diameter restricted Hausdorff dimensions take every value between its Hausdorff dimension and packing dimension. For the generalized intermediate dimensions Banaji \cite{Ban} proved similar type of interpolation results for all compact subsets of $\R^n$.
We show that similar property cannot 
hold for the $D$-diameter restricted Hausdorff dimensions by constructing a compact set $K\subset \R$ such that the set of $D$-diameter restricted Hausdorff dimensions of $K$ is $\{0,1\}$.

Banaji suggested that perhaps Question~\ref{q:Fraser} can be saved by adding the following property, which is well known for the Hausdorff dimension:
\begin{enumerate}\setcounter{enumi}{7}
\item \label{i:holder}
H\"older-stability: 
$\dim f(A)\le \frac{1}{\alpha}\dim A$ for each $\alpha\in (0,1]$ and $\alpha$-H\"older map $f\colon \R^n\to \R^n$.
\end{enumerate}
Note that the implications $\eqref{i:holder}\Rightarrow \eqref{i:lip}\Rightarrow \eqref{i:bilip}$ and $\eqref{i:sigma}\Rightarrow \eqref{i:stable}$ clearly hold. Thus we obtain the following modification of Question~\ref{q:Fraser}:

\begin{question}[Banaji--Fraser]\label{q:BanajiFraser}
Assume that $\dim \colon \iK(\R^n)\to [0,n]$ is a function satisfying the following properties:
\begin{enumerate}[(i)] 
\item
agrees with the similarity dimension on each self-similar set satisfying the strong separation condition,
\item  
monotonicity: $A\su B\Longrightarrow \dim A\le \dim B$,
\item $\sigma$-stability: $\dim(\cup_{i=1}^\infty A_i)=\sup_i\dim A_i$,
\item measurability: $\dim\colon \iK(\R^n)\to [0,n]$ is a Borel function, 
\item H\"older-stability: 
$\dim f(A)\le \frac{1}{\alpha}\dim A$ for each $\alpha\in (0,1]$ and $\alpha$-H\"older map $f\colon \R^n\to \R^n$.
\end{enumerate}
Does it follow that $\dim$ is the Hausdorff dimension?
\end{question}

We show that for $D$-diameter restricted Hausdorff dimensions the answer is positive:
among them only those have the H\"older-stability that agree with the Hausdorff dimension.
We remark that Siqi Wang \cite{SW} has very recently given a negative answer to Question~\ref{q:BanajiFraser} by
combining different $D$-diameter restricted Hausdorff dimensions.

We also study the supremum of the $D$-diameter restricted Hausdorff dimensions, which we denote by $\rdim$ and call the  
\emph{universal restricted Hausdorff dimension}. We show that supremum cannot be replaced by maximum even for compact sets. 
We also show that this dimension has the above properties (1-6) but does not have property (7). 
We prove that $\rdim$ differs from all the above mentioned old, more recent, and new dimensions.
It also turns out that for $\sigma$-compact sets $\rdim$ is in between the lower packing dimension and the packing dimension.

The same way as for the Hausdorff dimension, one can easily show that for any scale set $D$, every set $E\su\R^n$ can be covered by a $G_\delta$ set of the same $D$-diameter restricted Hausdorff dimension. It turns out that this does not hold for $\rdim$. In fact, every set $E\su\R^n$ which is co-meager in 
a non-empty open set contains a compact set $K$ with $\rdim K=n$. In particular, each non-meager set $E\su\R^n$ with the Baire property satisfies $\rdim E=n$, and no dense set can be covered by a $G_\delta$-set of zero universal restricted Hausdorff dimension.

Our motivation to define these new dimensions was to answer Question~\ref{q:Fraser} but they can also have other applications. 
They can serve as useful invariants.
For example, a possible way to prove that two given sets are not bilipshitz invariant can be to construct a scale set $D$ for which the $D$-diameter restricted Hausdorff dimensions of the two sets differ.

\subsection*{Structure of the paper}
We recall some known definitions and facts in Section~\ref{s:prel}.
Section~\ref{s:dimrd} contains the definition of the $D$-diameter restricted Hausdorff dimension and its properties, including the negative answer to Questions~\ref{q:Fraser} of Fraser.
Section~\ref{s:urhd} is 
about the universal restricted Hausdorff dimension. Finally, in Section~\ref{s:bilipdim} we show the optimal bounds for bilipschitz invariant, reasonably nice dimensions. 
Section~\ref{s:bilipdim} is independent of Sections~\ref{s:dimrd} and \ref{s:urhd}.

\section{Preliminaries}\label{s:prel}

\subsection{Notation}
We denote the set of positive integers by $\N$. For a subset $E$ of a metric space let $|E|$ be its diameter. For a finite set $F$ let us denote by $\# F$ its cardinality.

\subsection{Classical dimensions}\label{ss:classical}
Let $(X,d)$ be a metric space. For $x\in X$ and $r>0$ we denote by $B(x,r)$ the closed ball of radius $r$ centered at $x$, and by $N(X,r)$ the minimal number
of closed balls of radius $r$ that cover $X$. We say that $S\subset X$ is a $\delta$-packing if $d(x,y)>\delta$ for any two distinct $x,y\in S$. For $n\in \NN$ define 
\begin{equation*}
M_{n}(X)=\max \{\#S: S\subset X \text{ is a $2^{-n}$-packing}\}. 
\end{equation*}
The \emph{upper} and \emph{lower box dimensions} of a bounded set $X\su\R^n$ are respectively defined as
\begin{equation*} 
\ubdim X=\limsup_{r \to 0+} \frac{\log N(X,r)}{\log (1/r)} \quad \text{and} \quad 
\lbdim X=\liminf_{r \to 0+} \frac{\log N(X,r)}{\log (1/r)}.
\end{equation*}
We obtain equivalent definitions as 
\begin{equation}\label{e:boxdimwithMn} 
\ubdim X=\limsup_{n \to \infty} \frac{\log M_n(X)}{n\log 2}
\quad \text{and} \quad 
\lbdim X=\liminf_{n \to \infty} \frac{\log M_n(X)}{n\log 2}.
\end{equation}
If $\ubdim X=\lbdim X$ then the common value $\bdim X$ is the \emph{box dimension} of $X$.
The \emph{packing} and \emph{lower packing dimensions} of a set $X\su\R^n$ can be defined as
the $\sigma$-stable modification of the upper and lower box dimensions, respectively:
\begin{align*}
    \pdim X&=\inf\left\{\sup_i \ubdim X_i \colon X=\bigcup_{i=1}^\infty X_i,\, X_i \textrm{ are bounded}\right\}, \\
\lpdim X&=\inf\left\{\sup_i \lbdim X_i \colon X=\bigcup_{i=1}^\infty X_i,\, X_i \textrm{ are bounded}\right\}.
\end{align*}
For the following standard lemma see \cite[Lemma~2.1]{BP2}. 
\begin{lemma} \label{l:lp} 
Let $X\subset \RR^n$. 
\begin{enumerate}[(i)]
\item \label{i:reg1} If $X$ is closed and any open set $V$ intersecting $X$ satisfies $\ubdim (X\cap V)\geq s$, then $\pdim X\geq s$.
\item \label{i:reg2} 
If $\pdim X>s$ then there exists a relatively closed subset $F\subset X$ such that $\pdim (F\cap V)>s$ for any open set $V$ intersecting $F$.  
\item \label{reg:i3} The analogues of both \eqref{i:reg1} and \eqref{i:reg2} hold with $\lbdim$ and $\lpdim$ in place of $\ubdim$ and $\pdim$, respectively. 
\end{enumerate}
\end{lemma}

For more on these dimensions and for the concepts of the \emph{Hausdorff dimension} $\hdim$ and the \emph{$s$-dimensional Hausdorff measure} $\iH^s$ see e.g.~the books \cite{Fa} and \cite{Ma}.

\subsection{Intermediate dimensions}\label{ss:intermediate}
The intermediate dimensions were introduced by Falconer, Fraser, and Kempton \cite{FFK}, and they are between the Hausdorff and the box dimensions. We will consider the more general notion of upper $\Phi$-intermediate dimensions introduced by Banaji \cite{Ban}, where we can get back the original (upper) intermediate dimensions with the special choice of functions $\Phi_{\theta}(\delta)=\delta^{1/\theta}$ with $0< \theta < 1$. 

A function $\Phi\colon (0,\infty)\to (0,\infty)$ is called \emph{admissible} if $\Phi$ is non-decreasing, we have $0<\Phi(\delta)\leq \delta$ for all $\delta>0$, and $\Phi(\delta)/\delta\to 0$ as $\delta\to 0+$. For a bounded set $X\subset \R^n$ we define the \emph{upper
$\Phi$-intermediate dimension} of $X$ by 
\begin{align*}
\PhiBdim X=\inf\Big\{& s\geq 0: \forall \eps>0~\exists \delta_0>0 \text{ such that } \forall \delta \in (0,\delta_0) \text{ there exists a cover}   \\
& \{U_i\}_{i\geq 1} \text{ of $X$ such that } \Phi(\delta)\leq |U_i|\leq \delta \text{ for all $i$ and } \textstyle{\sum_i} |U_i|^s\leq \eps\Big\}.
\end{align*}
Following Douzi and Selmi \cite{DS}, we define the \emph{modified upper $\Phi$-intermediate dimension} of 
a set $X\subset \RR^n$ by 
\begin{equation*} 
\PhiMBdim X=\inf\left\{ \sup_i \PhiBdim X_i: X=\bigcup_{i=1}^{\infty} X_i,\, X_i \text{ are bounded}\right\}. 
\end{equation*}

\begin{remark} If $\Phi$ satisfies $\frac{\log \delta}{\log \Phi(\delta)}\to 0$ as $\delta\to 0+$, then by \cite[Proposition~3.4]{Ban} we obtain that $\PhiBdim X=\ubdim X$ for all $X$, so in this case $\PhiMBdim$ is the packing dimension. 
\end{remark}

The following lemma easily follows from the Baire category theorem, for the argument see e.g.~\cite[Lemma~2.8.1\,(i)]{BP}.

\begin{lemma} \label{l:Baire} Let $\Phi$ be an admissible function and let $X\subset \RR^n$ be a closed set. If $\PhiBdim(X\cap U)\geq \alpha$ for any bounded open set $U$ with $X\cap U\neq \emptyset$, then $\PhiMBdim(X)\geq \alpha$.
\end{lemma}

\subsection{Assouad-type dimensions}\label{ss:assouad}
Assouad-type dimensions have become very popular in fractal geometry as well, for example Fraser's monograph \cite{F} is entirely dedicated to them. The Assouad dimension was used in P.~Assouad's PhD thesis \cite{Assouad} in the context of embedding problems, but it can even traced back to Bouligand \cite{Bou}. The natural dual notion of the Assouad dimension is the lower dimension, which was introduced by Larman in \cite{L}, where it was called the minimal dimensional number. 

The \emph{Assouad dimension} of a set $X\subset \RR^n$ is defined by
\begin{align*} 
\adim X=\inf\Big\{&\alpha : \text{there is a } C > 0 \text{ such that for all } 0 < r < R \text{ and}
\\ &\text{for all } x\in X \text{ we have } N(B(x,R)\cap X,r) \leq C \left(\frac{R}{r}\right)^\alpha\Bigg\}. 
\end{align*}
The \emph{lower dimension} of a set $X\subset \RR^n$ is defined as 
\begin{align*} 
\ldim X = \sup\Big\{&\alpha : \text{there is a } C > 0 \text{ such that for all } 0 < r < R \le \diam X
\\ &\text{and for all } x  \in X \text{ we have } N(B(x,R)\cap X,r) \ge C \left(\frac{R}{r}\right)^\alpha\Bigg\}. 
\end{align*}
Fraser and Yu \cite{FY} introduced the \emph{modified lower dimension} $\mldim$ by making the lower dimension monotonic as 
\begin{equation*}
\mldim X=\sup\{\ldim E: E\subset X\}.
\end{equation*}

\subsection{Inequalities between the dimensions} Let $\Phi$ be an arbitrary admissible function. The following theorem summarizes the inequalities between the above dimensions, see \cite{F} and the references above.  
\begin{claim} \label{c:ineq}
For any bounded set $X\subset \RR^n$ we have
\begin{equation*} 
\hdim X\leq \lbdim X \leq \ubdim X\leq  \adim X \text{ and } \hdim X\leq \PhiBdim X \leq \ubdim X;
\end{equation*}
and any set $X\subset \RR^n$ satisfies
\begin{equation*}
\hdim X \leq \lpdim X \leq \pdim X\leq  \adim X \text{ and } \hdim X \leq \PhiMBdim X \leq \pdim X.
\end{equation*}
We have $\ldim X\leq \mldim X$ and for closed sets $X$ also $\mldim X\leq \hdim X$.
\end{claim}

\subsection{Some basic properties of the dimensions}

Property (1), and bilipschitz invariance (5) of the Introduction hold for all of the above dimensions, monotonicity (2) holds for all of them except the lower dimension, while finite stability (3) holds for all of them except the lower box dimension and the lower dimension. The finite stability of the modified lower dimension can be found in \cite{BEK}. Lipschitz stability (4) is satisfied for all the above classical and intermediate dimensions, but fails to hold for the Assouad-type dimensions. The $\sigma$-stability (6) holds for the Hausdorff dimension, the packing dimension, the lower packing dimension, the modified upper $\Phi$-intermediate dimensions, but it fails to hold for the others, see \cite{F} and the references above.

\subsection{H\"older maps and bilipschitz equivalence}
Let $(X,d_X)$ and $(Y,d_Y)$ be metric spaces and $f\colon X\to Y$. We say that $f$ is $\alpha$-H\"older if there exists a finite constant $C$ such that for all $x,z\in X$ we have 
\begin{equation*} d_Y(f(x),f(z))\leq C \left(d_X(x,z)\right)^{\alpha}. 
\end{equation*} 
We call $f$ \emph{Lipschitz} if it satisfies the above inequality with $\alpha=1$. The metric spaces $X$ and $Y$ are said to be \emph{bilipschitz equivalent} if there exists a bijection $f\colon X\to Y$ such that both $f$ and its
inverse are Lipschitz.

\subsection{Self-similar sets, ultrametric spaces and Ahlfors--David regular sets}

Let $K\su \R^n$ 
be a self-similar set with 
contracting similarity maps $\{f_i\}_{1\leq i\leq m}$,
that is, $K=\bigcup_{i=1}^m f_i(K)$, and let
$r_i$ be the similarity ratio of $f_i$. 
The unique $s$ for which $r_1^s+\ldots+r_m^s=1$ is called the \emph{similarity dimension} of $K$.
If $r_1=\ldots=r_m$ then $K$ is said to be \emph{homogeneous}.
We say that $K$ satisfies the \emph{strong separation condition} (SSC) if $f_i(K)\cap f_j(K)=\emptyset$ for all $1\leq i< j\leq m$. 

A metric space $(X,d)$ is called \emph{ultrametric} if the triangle inequality is replaced with the stronger inequality
\begin{equation*} 
d(x,y)\leq \max\{d(x,z),d(y,z)\} \quad  \textrm{for all } x,y,z\in X.
\end{equation*}

This is equivalent to the property that if $a\leq b\leq c$ are sides of a triangle in $X$, then $b=c$. 
It is well known and easy to prove that every self-similar set with the SSC is bilipschitz equivalent to an ultrametric space.

Recall that for any $0 < s \leq n$ a Borel set $T \su \R^n$ is said to be \emph{Ahlfors--David $s$-regular}, or shortly 
\emph{$s$-regular} if 
there exist $c,C\in (0,\infty)$ such that 
\begin{equation*} 
c  r^{s} \leq \iH^{s}(T \cap B(t,r)) \leq C r^{s}
\end{equation*}
for all $t \in T$ and $0<r < \diam T$.

\subsection{Baire category}\label{ss:Baire}

Recall that a set $E\su\R^n$ is said to be \emph{nowhere dense} if every ball $B\su\R^n$ contains a ball $D$ disjoint to $E$.
A set is \emph{meager} if it can be obtained as a countable union of nowhere dense sets. 
We say that a set $E\su\R^n$ is \emph{co-meager} in a set $U\in\R^n$ if $U\sm E$ is meager.
By the Baire category theorem every dense $G_\delta$ set is co-meager in $\R^n$.

A set is said to have the \emph{Baire property}
if its symmetric difference with an open set is meager. In particular, by the above definitions every non-meager set with the Baire property is co-meager in a nonempty open set. 
It is well known that every Borel set has the Baire property.
For more on these notions and results see \cite{Kec}.

\section{Diameter restricted Hausdorff dimensions}
\label{s:dimrd}

Recall that on $\R$ the \emph{Sorgenfrey} topology $\tau_S$ is the topology generated by intervals of the form $[a,b)$. 

\begin{defi} \label{d:restdim}
Let
\begin{equation*}
\iD=\{ D\su (0,\infty) \colon D\in\tau_S, \inf D=0\}.    
\end{equation*}
Fix $D\in \iD$ and a positive integer $n$. First for $s\geq 0$ and $E\su \R^n$ we define the \emph{$s$-dimensional $D$-diameter restricted Hausdorff content} 
the same way as the Hausdorff content is defined, the only difference is that in the covering we allow only sets with diameter in $D$: 
\begin{equation*}
\iH^s_{D,\infty}(E)=\inf \left\{ \sum_{i=1}^\infty |E_i|^s: E\su \bigcup_{i=1}^{\infty} E_i,~|E_i|\in D \textrm{ for all $i$} \right\}.
\end{equation*}
Now we define the \emph{$D$-diameter restricted Hausdorff dimension} of $E$ 
the usual way:
\begin{equation*}
\drdim E = \inf \left\{ s\geq 0: \iH^s_{D,\infty}(E)=0\right\}.   
\end{equation*}
\end{defi}

\begin{remark} \label{r:wlogD}
We may assume without loss of generality that $D\in \iD$ is of the form
$D=\cup_{k=1}^{\infty} [2^{-n_k}, 2^{-n_k+1})$, where $n_k$ is a strictly increasing sequence of positive integers. Indeed, if $D$ is not of this form then let $n_k$ be the $k$th positive integer for which $D\cap [2^{-n_k}, 2^{-n_k+1})\neq \emptyset$, and let $\widetilde{D}=\cup_{k=1}^{\infty} [2^{-n_k}, 2^{-n_k+1})$. Then it is easy to see that $\drtdim E=\drdim E$ for all $E\subset \R^n$. Therefore, we basically calculate the Hausdorff dimension along a given scale sequence.
\end{remark}

\begin{remark} Note that we can extend the definition $\drdim E$ to all metric spaces $E\su X$ in an intrinsic way by changing the content as
\begin{equation*}
\widetilde{\iH}^s_{D,\infty}(E)=\inf \left\{ \sum_{i=1}^\infty |r_i|^s: E\su \bigcup_{i=1}^{\infty} B(x_i,r_i),~ x_i\in E \textrm{ and } r_i\in D \textrm{ for all $i$} \right\}.
\end{equation*}
\end{remark}

For any $D\in\iD$ the $D$-diameter restricted Hausdorff dimension $\drdim$ has the following basic properties. 
Most of them are the same as the basic properties of the Hausdorff dimension, and the simple proofs are also analogous.

\begin{prop}\label{p:easydrdim}
For any $D\in\iD$ the dimension $\drdim$ is
\begin{enumerate}[(i)]
\item \label{pD:1} Lipschitz-stable,
\item \label{pD:2} monotone,   
\item \label{pD:3} $\drdim S=\hdim S$ for any self-similar set $S$,
\item \label{pD:4} $\sigma$-stable, 
\item \label{pD:5} every $E\su\R^n$ can be covered by a $G_\delta$ set $B$ with $\drdim B=\drdim E$,
\item \label{pD:6} $\drdim\colon \iK(\RR^n)\to [0,n]$ is a Borel function, and
\item \label{pD:7} $\hdim E\le \drdim E\le \pdim E$ for any $E\su\R^n$.
\end{enumerate}
\end{prop}

\begin{proof}
Properties \eqref{pD:1}, \eqref{pD:2}, \eqref{pD:4}, \eqref{pD:5} and \eqref{pD:6} can be proved the same way as for the Hausdorff dimension.

The first inequality of \eqref{pD:7} is clear from the definition. For $\drdim E\leq \pdim E$ by \eqref{pD:4} it is enough to prove that $\drdim E\leq \ubdim E$ for bounded sets $E$. Assume that $E$ is bounded and $s>\ubdim E$, then $E$ has a good cover with arbitrarily small balls of equal size. Thus $\iH^s_{D,\infty}(E)=0$, which implies $\drdim E\leq s$, so $\drdim E\leq \ubdim E$.

Since $\hdim S=\pdim S$ for all self-similar sets, \eqref{pD:7} implies \eqref{pD:3}.
\end{proof}

\begin{lemma}[Mass distribution principle for the $D$-diameter restricted Hausdorff dimensions] \label{l:mass}
Let $D\in \iD$ and $E\subset \R^n$. Assume that $\mu$ is a mass distribution supported within $E$ such that there are constants $C<\infty$ and $s\geq 0$ such that for each closed set $U\subset \R^n$ with $|U|\in D$ we have
\begin{equation*} \mu(U)\leq C|U|^s.
\end{equation*}
Then $\drdim E\geq s$.
\end{lemma}

\begin{proof}
Assume that $\{U_i\}_{i\geq 1}$ is a cover of $E$ such that $|U_i|\in D$ for all $i$. Let $V_i$ be the closure of $U_i$, then clearly $|V_i|=|U_i|$. Then
\begin{equation*}
0<\mu(E)\leq \mu\left(\bigcup_{i=1}^{\infty} V_i\right)\leq  
\sum_{i=1}^{\infty} \mu(V_i)\leq C\sum_{i=1}^{\infty} |U_i|^s.
\end{equation*}
Therefore $\iH^s_{D,\infty}(E)\geq \frac{\mu(E)}{C}$, so indeed $\drdim E\geq s$.
\end{proof}

\begin{notation}
For any $S\su\N$ we define a `uniformly branching set' defined by digit restrictions as follows, see e.g.~\cite[Example 1.3.2]{BP}:  
\begin{equation}\label{e:ASdef}
A_S=\left\{\sum_{i\in S} x_i 2^{-i} \colon x_i\in\{0,1\}\right\}, 
\end{equation}
and also let
\begin{equation}\label{e:MSdef}
 M_S(k)=\frac{\#(S\cap\{1,\ldots,k\})}{k}.   
\end{equation}
\end{notation}
It is well known (see e.g.~\cite{BP}) that for
any $S\su\N$ we have
\begin{align}
\label{e:pdimAS}
\pdim A_S=\ubdim A_S &=\limsup_{k\to\infty} M_S(k),  \\
\label{e:hdimAS}
\hdim A_S =\lbdim A_S &=\liminf_{k\to\infty} M_S(k).
\end{align}
Assume that $S_i\subset \N$ for all $i\geq 1$ and define
\begin{equation} \label{d:ASi}
A(\{S_i: i\geq 1 \})=\{0\}\cup \bigcup_{i=1}^{\infty} 2^{-i}(1+A_{S_i}),
\end{equation}
where we use the notation $aX=\{ax\colon x\in X\}$ and $a+X=\{a+x\colon x\in X\}$.

Let $(n_k)$ be a strictly increasing sequence of positive integers and let 
\begin{equation*}
D=\bigcup_{k=1}^\infty [2^{-n_k},2^{-n_k+1}).
\end{equation*}

\begin{claim}\label{c:dimofAS}
Using the above notation we have 
\begin{equation}\label{e:drdimofAS}
\drdim A_S =\liminf_{k\to\infty} M_S(n_k),\text{ and }
\end{equation}
\begin{equation} \label{e:dimASi}
\drdim A(\{S_i: i\geq 1 \})=\sup_{i\geq 1} \liminf_{k\to\infty} M_{S_i}(n_k).
\end{equation}
\end{claim}

\begin{proof}
  By the $\sigma$-stability of $\drdim$, \eqref{e:drdimofAS} implies \eqref{e:dimASi}, so it is enough to prove \eqref{e:drdimofAS}. The set $A_S$ can be clearly covered by $2^{n_k M_S(n_k)}$ intervals of length $2^{-n_k}$,
which shows that 
$\drdim A_S \le \liminf_{k\to\infty} M_S(n_k)$.

Let $s=\liminf_{k\to\infty} M_S(n_k)$.
By Lemma~\ref{l:mass}, to prove the reverse inequality it is enough to construct a mass distribution $\mu$ supported on $A_S$ with a finite constant $C$ such that 
\begin{equation} \label{e:mu}
\mu(U)\le C|U|^s  \text{ for every closed set } U\su\R. 
\end{equation}
Let $\mu$ be the natural uniform mass distribution supported on $A_S$: for each $j\geq 1$ our set $A_S$ is a union of $2^{jM_S(j)}$ isometric compact sets of diameter at most $2^{-j}$, let us assign mass $2^{-jM_S(j)}$ to all of them. It is easy to check that this $\mu$ indeed satisfies \eqref{e:mu}, which concludes the proof.
\end{proof}

Now we can prove the following Darboux-type result for the diameter restricted Hausdorff dimensions of the sets of the form \eqref{e:ASdef}.   

\begin{theorem}\label{t:ASDarboux}
Let $S\subset \NN$. Then
\begin{equation*}
\left\{\drdim A_S: D\in \iD\right\}=[\hdim A_S, \pdim A_S].
\end{equation*}
\end{theorem}

\begin{proof}
By Proposition~\ref{p:easydrdim}~\eqref{pD:7} we have $\{\drdim A_S: D\in \iD\}\subset [\hdim A_S, \pdim A_S]$. The inclusion $[\hdim A_S, \pdim A_S]\subset \{\drdim A_S: D\in \iD\}$  follows from Claim~\ref{c:dimofAS} by using the fact that for any bounded sequence $a_k$ and any $\alpha\in[\liminf a_k,\limsup a_k]$ there is a subsequence $a_{n_k}$ that converges to $\alpha$.
 \end{proof}

One can ask whether a similar Darboux property holds for every compact set $K\subset \R^n$. The following result shows that the answer is negative, we give a counterexample of the form \eqref{d:ASi}.

\begin{prop}
There exists a compact set $K\su\R$ of the form \eqref{d:ASi} for which 
\begin{equation*}
\left\{\drdim K \colon D\in \iD\right\} = \{0,1\}.
\end{equation*}
\end{prop}

\begin{proof} Let $K=A(\{S_i: i\geq 1 \})$ where $S_i$ are defined as follows. For $i\geq 1$ let 
\begin{equation*}
S_i=\bigcup_{m=i}^{\infty} \left\{\frac 1i((3m)!), \frac 1i((3m)!)+1,\dots,i((3m)!)\right\}.
\end{equation*}
By \eqref{e:hdimAS}, we have $\hdim A_{S_i}=\liminf_{k} M_{S_i}(k)=0$ for all $i\geq 1$, so for $D=(0,\infty)$ we have $\drdim K=\hdim K=0$.
By Theorem~\ref{t:ASDarboux} and \eqref{e:pdimAS} we obtain that
\begin{equation}\label{e:nonzerowitness}  
\sup\left\{\drdim A_{S_i}\colon D\in\iD\right\}=
\pdim A_{S_i}= \limsup_{k\to \infty} M_{S_i}(k)=1 -\frac{1}{i^2}.
\end{equation}

Fix $D\in\iD$.
If $\drdim A_{S_i}=0$ for all $i\geq 1$, then by the $\sigma$-additivity of $\drdim$ we clearly have  $\drdim K=0$. Now assume that there is a positive integer $\ell$ such that $\drdim A_{S_{\ell}}>0$. This is indeed possible since by \eqref{e:nonzerowitness} such $D$ exists. It remains to prove that in this case $\drdim K=1$. 

By Remark~\ref{r:wlogD}, we may assume that
\begin{equation*}
D=\bigcup_{k=1}^\infty [2^{-n_k},2^{-n_k+1}),
\end{equation*}
where $n_k$ is a strictly increasing sequence of positive integers. 
Then \eqref{e:drdimofAS} and $\drdim A_{S_{\ell}}>0$ imply that we can
choose a positive integer $j$ such that for all large enough $k$ we have 
\begin{equation} \label{e:1/j}
M_{S_{\ell}}(n_k)>\frac 1j,
\end{equation} 
and the definition of $S_{\ell}$ and \eqref{e:1/j} imply that there is an $m(k)\in\N$ such that
\begin{equation} \label{e:nkin} 
n_k\in \left\{\frac 1\ell((3m(k))!),\dots, \ell j ((3m(k))!)\right\}.
\end{equation}
Thus we obtained that 
for all large enough $k$ there exists a positive integer $m(k)$ such that 
\eqref{e:nkin} holds. Since for all $i\geq j$ and large enough $k$ we have 
\begin{equation} \label{e:subS}
\left\{\frac {1}{i\ell}((3m(k))!),\dots, \ell j ((3m(k))!)\right\}\subset S_{i\ell},
\end{equation} 
\eqref{e:nkin} and \eqref{e:subS} imply that
\begin{equation} \label{e:Sli} 
\liminf_{k\to \infty} M_{S_{i \ell}}(n_k)\geq 1-\frac 1i.
\end{equation}
Therefore \eqref{e:dimASi} and \eqref{e:Sli} yield
\begin{equation*} \drdim K = \sup_{i\geq 1} \liminf_{k\to\infty} M_{S_{i}}(n_k)=1,
\end{equation*}
and the proof is complete.
\end{proof}

As a consequence of Theorem~\ref{t:ASDarboux}, we can give a negative answer to Question~\ref{q:Fraser}
of Fraser.

\begin{cor}
There exist continuum many different dimensions among the $D$-restricted Hausdorff dimensions that differ from the Hausdorff dimension and satisfy all the properties (1)-(7) of the Introduction. 
\end{cor}

\begin{proof}
Fix $S\su\N$ for which the lower and upper densities differ, 
so \eqref{e:pdimAS} and \eqref{e:hdimAS} imply $\hdim A_S<\pdim A_S$.
By Theorem~\ref{t:ASDarboux} for any $\alpha\in (\hdim A_S, \pdim A_S]$ there is a $D=D(\alpha)\in \iD$ such that $\drdim A_S=\alpha$. These dimensions satisfy all the properties (1)--(7) of the Introduction by Proposition~\ref{p:easydrdim}.  
The set $A_S$ witnesses that these dimensions are distinct and differ from the Hausdorff dimension.
\end{proof}

One can ask if the $D$-diameter restricted Hausdorff dimensions differ from the existing notions of dimensions. 
As we already saw one of the continuum many of them is the Hausdorff dimension.
The next result excludes all other dimensions of Section~\ref{s:prel}.

\begin{prop}
    For any $D\in\iD$ the $D$-restricted Hausdorff dimension can agree only with the Hausdorff dimension among the dimensions defined or mentioned in Section~\ref{s:prel}.
\end{prop}

\begin{proof}
We check that among the dimensions defined or mentioned in Section~\ref{s:prel}, only the Hausdorff dimension satisfies all the properties of the $D$-restricted Hausdorff dimensions obtained in Proposition~\ref{p:easydrdim}. Since among them only the Hausdorff dimension, the lower packing dimension, the packing dimension, and the modified upper $\Phi$-intermediate dimensions are $\sigma$-stable (\eqref{pD:4} of Proposition~\ref{p:easydrdim}), it is enough to consider those. But among those only the Hausdorff dimension satisfies \eqref{pD:5} of Proposition~\ref{p:easydrdim}, since the others take maximal value on dense $G_{\delta}$ sets.
\end{proof}

We saw in Proposition~\ref{p:easydrdim}
that every $D$-diameter restricted Hausdorff dimension is Lipschitz stable (property \eqref{i:lip} of the Introduction).  
The next result shows that among them only the Hausdorff dimension is H\"older stable (property \eqref{i:holder} of the Introduction). 
This shows that none of these dimensions are counterexamples to Question~\ref{q:BanajiFraser}.

\begin{theorem} Let $D\in \iD$. Assume that for all $\alpha\in (0,1]$, $\alpha$-H\"older functions $f\colon \R\to \R$ and compact sets $K\subset [0,1]$ we have 
\begin{equation}\label{e:Holderdim} 
\drdim f(K)\leq \frac{1}{\alpha} \drdim K.
\end{equation} 
Then $\drdim=\hdim$.
\end{theorem}

\begin{proof}
By Remark~\ref{r:wlogD} we may assume that $n_k$ is a strictly increasing sequence of integers such that
\begin{equation*} 
D=\bigcup_{k=1}^{\infty} [2^{-n_k},2^{-n_k+1}).
\end{equation*}
First we claim that 
\begin{equation*} \drdim=\hdim \quad \text{ if } 
\quad \lim_{k\to \infty} \frac{n_{k+1}}{n_k}=1. 
\end{equation*}
Indeed, whenever a set $E_i$ from a cover of $E$ satisfies $|E_i|\in [2^{-n_{k+1}+1}, 2^{-n_k})$, just replace it with a set $E_i\subset \widetilde{E}_i$ for which $|\widetilde{E}_i|=2^{-n_k}$. 
Then it is easy to see that the new covers witness $\drdim E\leq \hdim E$, 
and so $\drdim E=\hdim E$ by Proposition~\ref{p:easydrdim}\eqref{pD:7}.

Therefore, we may assume that $\limsup_{k} \frac{n_{k+1}}{n_k}>1$. Choose $\iK\subset \N$ such that 
\begin{equation*}
\lim_{k\in \iK} \frac{n_{k+1}}{n_k}=\limsup_{k\to \infty} \frac{n_{k+1}}{n_k}>1.
\end{equation*}
Define $S\subset \N$ such that
\begin{equation*}  
S=\bigcup_{k\in \iK} \left\{ \left\lceil\frac{n_k+n_{k+1}}{2}\right \rceil,\dots, n_{k+1}-2\right\}\cup \bigcup_{k\in \N\setminus \iK} \left(2\N\cap \{n_k,\dots,n_{k+1}-1\}\right).
\end{equation*} 
Using the notation \eqref{e:MSdef} we obtain that 
\begin{equation} \label{e:MSi}
M_S(i)\leq \frac 12 \text{ for all $i\in \N$}. 
\end{equation} 
The definition of $A_S$ in \eqref{e:ASdef} and equation \eqref{e:drdimofAS} imply that 
\begin{equation} \label{e:dimAS}
\drdim A_S=\liminf_{k\to \infty} M_S(n_k)=\frac 12.
\end{equation} 
For positive integers $p<q$ define 
\begin{equation*} 
\varphi_{p,q}\colon \N\to\N, \quad \varphi_{p,q}(kp-m)=kq-m 
\end{equation*}
for all positive integers $k$ and integers $m$ with $0\leq m\leq p-1$. Observe that for all $i\in \N$ we have 
\begin{equation} \label{e:fbound} 
\frac{i-p}{r}<\varphi_{p,q}(i)\leq \frac ir, \text{ where } r=\frac pq.
\end{equation} 

We define the continuous onto map
\begin{equation*}
f_{p,q}\colon A_{\varphi_{p,q}(S)}\to A_S, \quad f_{p,q}\left(\sum_{i=1}^{\infty} x_{i}2^{-\varphi_{p,q}(i)}\right)=\sum_{i=1}^{\infty} x_{i}2^{-i}.
\end{equation*} 
It is not hard to check using the upper bound in \eqref{e:fbound} that $f_{p,q}$ is $\frac pq$-H\"older. Note that each function $f_{p,q}$ can be extended to $\R$ as a $\frac pq$-H\"older function by a well-known extension result \cite[Corollary~1]{Sh}, so $f_{p,q}$ need to satisfy \eqref{e:Holderdim}. However, it is easier for us to keep the original domains for the rest of the proof.

First assume that $\lim_{k\in \iK} \frac{n_{k+1}}{n_k}=\infty$. 
In this case let $T=\varphi_{1,2}(S)$ and note that $\varphi_{1,2}(n)=2n$. 
Then $f_{1,2}\colon A_T\to A_S$ is $\frac 12$-H\"older and onto. 
Since $T\cap [2n_k,n_{k+1}]=\emptyset$ for all $k\in \iK$, we easily obtain that 
\begin{equation*} 
\drdim A_T=\liminf_{k\to \infty} M_T(n_k)\leq \liminf_{k\in \iK} \frac{2n_k}{n_{k+1}}=0, 
\end{equation*}
which together with \eqref{e:dimAS} contradicts \eqref{e:Holderdim}.

Finally, assume that 
\begin{equation}\label{e:limnk}
   \lim_{k\in \iK} \frac{n_{k+1}}{n_k}=1+2\eps 
\end{equation}
 for some $0<\eps<\infty$. Let $r=\frac pq\in (0,1)$ be a rational number such that 
\begin{equation} \label{e:r} \frac{1}{1+2\eps}<r<\frac{1+\eps}{1+2\eps}.
\end{equation} 
Let $T=\varphi_{p,q}(S)$, then the map $f_{p,q}$ is $r$-H\"older and onto. The essence of the proof will be that the allowed diameters coming from $D$ are not optimal for covering $A_S$, it can be covered more efficiently with sets having diameter $2^{-m_k}$, where $m_k= \lceil(n_k+n_{k+1})/2\rceil$ and $k\in \iK$. Then $\varphi_{p,q}(m_k)\approx n_{k+1}$ will imply that the scale sequence $n_k$ works much better for the set $A_T$ than for $A_S$. This makes $\drdim A_T$ small enough to 
contradict the following special case of the assumption \eqref{e:Holderdim}:
\begin{equation} \label{e:ATHolder}
\drdim A_S\leq \frac 1r \drdim A_T.
\end{equation}
By \eqref{e:drdimofAS}, $\drdim A_T=
\liminf_{k\to \infty} M_T(n_k)$, so in order to get a contradiction with \eqref{e:ATHolder} and \eqref{e:dimAS}, it is enough to prove that 
\begin{equation} \label{e:AT} 
\liminf_{k\to \infty} M_T(n_k)<\frac r2.
\end{equation}

Assumption \eqref{e:limnk} yields 
\begin{equation} \label{e:nk+1} \lim_{k\in \iK} \frac{n_k+n_{k+1}}{2n_{k+1}}=\frac{1+\eps}{1+2\eps}.
\end{equation}
Thus for large enough $k\in \iK$ the lower bound in \eqref{e:fbound}, the upper bound of $r$ in \eqref{e:r}, and \eqref{e:nk+1} imply that
\begin{equation*} 
\varphi_{p,q}\left(\left\lceil \frac{n_k+n_{k+1}}{2}\right\rceil \right)>
\frac{\left\lceil \frac{n_k+n_{k+1}}{2}\right\rceil-p}{r}>
n_{k+1},
\end{equation*}
which means that 
\begin{equation*} \label{e:nk} \left(\varphi_{p,q}(n_k), n_{k+1}\right]\cap T=\emptyset.
\end{equation*} 
Hence \eqref{e:MSi} yields that 
\begin{align} \label{e:TS}
\begin{split}
\#\{T\cap \{1,\dots,n_{k+1}\}\}&\leq 
\#\{T\cap \{1,\dots,\varphi_{p,q}(n_k)\}\}\\
&=\#\{S\cap \{1,\dots,n_k\}\}\leq \frac{n_k}{2}. 
\end{split}  
\end{align}
Inequality \eqref{e:TS}, assumption \eqref{e:limnk} and the lower bound of $r$ in \eqref{e:r} imply that
\begin{equation*} \liminf_{k\to \infty} M_T(n_{k+1})\leq  \lim_{k\in \iK} \frac{n_k}{2n_{k+1}}<\frac r2,
\end{equation*}
thus \eqref{e:AT} holds. The proof of the theorem is complete.
\end{proof}

\section{The universal restricted Hausdorff dimension}
\label{s:urhd}

From the continuum many new dimensions we defined any studied in the previous section we can define a single one in the following natural way.

\begin{defi}
    The \emph{universal restricted Hausdorff dimension} is defined as
\begin{equation}\label{e:rdimsup}
    \rdim E = \sup\left\{\drdim E \colon D\in\iD\right\},
\end{equation}
or equivalently as
\begin{equation*}
  \rdim E = \inf \left\{ s : \iH^s_{D,\infty}(E)=0 \textrm{ for all $D\in \iD$} \right\}.   
\end{equation*}  
\end{defi}

Proposition~\ref{p:easydrdim} directly implies the following properties of this new dimension. 
 
\begin{prop}\label{p:easyrdim}
The universal restricted Hausdorff dimension $\rdim$ is
\begin{enumerate}[(i)]
\item \label{pR:1} Lipschitz-stable,
\item \label{pR:2} monotone, 
\item \label{pR:3} $\rdim S=\hdim S$ for any self-similar set $S$,
\item \label{pR:4} $\sigma$-stable and
\item \label{pR:7} $\hdim E\le \drdim E\le \rdim E\le \pdim E$ for any $E\su\R^n$ and $D\in\iD$.
\end{enumerate}
\end{prop}

The next theorem gives a better lower bound for $\rdim E$ if $E\subset \RR^n$ is $\sigma$-compact. 

\begin{theorem} \label{t:lpacking} 
For any $\sigma$-compact set $E\subset \RR^n $ we have 
\begin{equation*} 
\lpdim E \leq \rdim E\leq \pdim E. 
\end{equation*}
\end{theorem}
\begin{proof} The upper bound holds for any $E\subset \RR^n$ by Proposition~\ref{p:easyrdim}~\eqref{pR:7}. Hence we only need to prove that $\rdim E\geq \lpdim E$. Since $E$ is $\sigma$-compact and both $\lpdim$ and $\rdim$ are countably stable, we may assume without loss of generality that $E$ is compact. We may also suppose that $\lpdim E>0$, and let us fix an arbitrary $s\in (0,\lpdim E)$. By Lemma~\ref{l:lp} we can choose a compact subset $K\subset E$ such that $\lpdim (K\cap V)>s$ for every open set $V\subset \RR^n$ which intersects $K$. Since $s$ is arbitrary, it is sufficient to prove that $\rdim K\geq s$, for which it is enough to define a scale set $D$ of the form
\begin{equation*}
D=\bigcup_{k=1}^\infty [2^{-n_k},2^{-n_k+1})
\end{equation*} 
satisfying $\iH^s_{D,\infty}(K)\geq 1$. 

Now we will construct our $D$ by defining the sequence of positive integers $n_k$. 
Since $\lbdim K>s$, using \eqref{e:boxdimwithMn} we can choose $n_1\in \NN$ such that $M_{n_1-2}(K)\geq \lceil 2^{sn_1} \rceil$. Let $N_1=\lceil 2^{sn_1} \rceil$ and let $\{x_1,\dots, x_{N_1}\}$ be a $2^{-n_1+2}$-packing of $K$. Assume by induction that the positive integers $n_j$ and $N_j=\lceil 2^{sn_j} \rceil$, and the points $x_{i_1\dots i_j}\in K$ for $i_1\in \{1,\dots,N_1\},\dots, i_j\in \{1,\dots,N_j\}$ are already defined for all $j\in \{1,\dots,k\}$. From now on all balls $B(x,r)$ are assumed to be within $K$. Then we define 
\begin{equation*}
n_{k+1}=\min\{n>n_k: M_{n-2}(B(x_{i_1\dots i_k},2^{-n_k-1}))\geq 2^{sn}~\forall i_1\leq N_1,\dots,i_k\leq N_k\}. 
\end{equation*} 
As $\lbdim B(x_{i_1\dots i_k},2^{-n_k-1})>s$ for all indices $i_1\leq N_1,\dots, i_k\leq N_k$, we obtain that $n_{k+1}$ is finite. Let $N_{k+1}=\lceil 2^{sn_{k+1}} \rceil$, and for all indices $i_1\leq N_1,\dots,i_k\leq N_k$ let us choose $x_{i_1\dots i_k i_{k+1}}\in B(x_{i_1\dots i_k},2^{-n_k-1})$ for every $i_{k+1}\in \{1,\dots,N_{k+1}\}$ such that $\{x_{i_1\dots i_k i_{k+1}}: 1\leq i_{k+1}\leq N_{k+1}\}$ forms a $2^{-n_{k+1}+2}$-packing. 
This defines the sequence $n_k$ and so $D$ as well. From the construction for all $k\geq 2$ and indices $i_1,\dots,i_k$ we have  
\begin{equation} \label{eq:B}
B(x_{i_1\dots i_k}, 2^{-n_k})\subset B(x_{i_1\dots i_{k-1}}, 2^{-n_{k-1}}). 
\end{equation} 
Also, for any $k\in \NN$ and indices $i_1\leq N_1,\dots, i_{k-1}\leq N_{k-1}$ and $V\subset \RR^n$ we obtain the implication 
\begin{equation} \label{eq:index}
|V|<2^{-n_k+1}~\Rightarrow~\text{there is at most one } i_k \text{ with } B(x_{i_1\dots i_k},2^{-n_k})\cap V\neq \emptyset. 
\end{equation}

Finally, we will prove that $\iH^s_{D,\infty}(K)\geq 1$. Assume to the contrary that there is a cover $K\subset \bigcup_{i=1}^{\infty} V_i$ such that $|V_i|\in D$ for each $i$ and $\sum_i |V_i|^s<1$. We may assume that all sets $V_i$ are open, and by the compactness of $K$ we can choose a finite cover $K\subset \bigcup_{i=1}^{N} V_i$. Hence there exists $\ell\in \NN$ such that 
\begin{equation*} \{|V_i|: i\leq N\}\subset \bigcup_{k=1}^{\ell}  [2^{-n_k},2^{-n_k+1}).
\end{equation*}
For all $k\in\{1,\dots, \ell\}$ let 
\begin{equation*} \iF_k=\{V_i: |V_i|\in [2^{-n_k},2^{-n_k+1})\}. 
\end{equation*} 
By \eqref{eq:index} for all $k\in \{1,\dots,\ell\}$ and indices $i_1\in \{1,\dots N_1\},\dots,i_{k-1}\in \{1,\dots, N_{k-1}\}$ we have the implication  
\begin{equation} \label{eq:impl} \#\iF_k<N_k~\Longrightarrow~\exists i_k\in \{1,\dots, N_k\} \text{ with } B(x_{i_1\dots i_k},2^{-n_k})\cap \left(\bigcup \iF_k\right)=\emptyset. 
\end{equation}
Now there are two cases. If there exists a $k\in \{1,\dots,\ell\}$ such that $\# \iF_k\geq N_k$, then clearly \begin{equation*} 
\sum_{V_i\in \iF_k}|V_i|^s\geq 2^{-sn_k}N_k\geq 2^{-sn_k} \cdot 2^{sn_k}=1,
\end{equation*}
which contradicts our assumption $\sum_i |V_i|^s<1$. Finally, assume that $\# \iF_k<N_k$ for all $k\in \{1,\dots,\ell\}$. Applying \eqref{eq:impl} $\ell$ times we can inductively define indices $i_1\in\{1,\dots,N_1\},\dots,i_{\ell}\in \{1,\dots, N_{\ell}\}$ such that for all $k\in \{1,\dots,\ell\}$ we have 
\begin{equation} \label{eq:ii}
B(x_{i_1\dots i_k},2^{-n_k})\cap \left(\bigcup \iF_k\right)=\emptyset. 
\end{equation}
Then \eqref{eq:ii} and \eqref{eq:B} imply that $B(x_{i_1\dots i_{\ell}},2^{-n_{\ell}})\cap V_i=\emptyset$ for all $1\leq i\leq N$, which contradicts the assumption $K\subset \bigcup_{i=1}^{N} V_i$. This completes the proof. 
\end{proof}

\begin{cor} \label{c:dim}
Assume that $E\subset \RR^n$ is a $\sigma$-compact set with $\lpdim E =\pdim E$. Then $\rdim E=\pdim E$.
\end{cor}

\begin{remark} For self-similar sets $K\subset \RR^n$ it is known that 
$\hdim K=\bdim K$, so $\rdim K=\hdim K$. We show that the situation is different even for self-affine sets. Assume that $M\subset [0,1]^2$ is a Bedford-McMullen carpet with a defining pattern $P\subset \{0,\dots,n-1\}\times \{0,\dots, m-1\}$ and $n>m$. McMullen \cite{Mc} proved that $\bdim M$ exists, and $\lpdim M=\bdim M=\pdim M$ easily follows from Lemma~\ref{l:lp}. McMullen \cite{Mc} also proved that $\hdim M<\pdim M$ whenever $P$ has non-uniform horizontal fibers. Hence Corollary~\ref{c:dim} yields $\rdim M>\hdim M$ in these cases.
\end{remark}

Theorem~\ref{t:ASDarboux} and \eqref{e:pdimAS}--\eqref{e:hdimAS} show that $\rdim K>\lbdim K=\lpdim K=\hdim K$ is even possible for `uniformly branching sets' $K$. The next theorem states that $\rdim$ does not coincide with the packing dimension, or more generally, with the modified upper $\Phi$-intermediate dimensions. 

\begin{theorem}\label{t:rdimnotpdim}
Let $n$ be a positive integer and let $\Phi$ be an arbitrary admissible function. Then there exists a compact set $K\su\R^n$ for which 
\begin{equation*} 
\rdim K=0 \quad \text{and} \quad  \PhiMBdim K=n.
\end{equation*} 
Consequently, $\pdim K=n$.
\end{theorem}

\begin{proof} 
Assume that $\{Q_i\}_{i\geq 1}$ is an enumeration of the closed dyadic subcubes of $[0,1]^n$ such that $|Q_i|\leq |Q_j|$ for all $i>j$. We say that a dyadic cube is of \emph{level $i$} if it has edge length $2^{-i}$. Let us define strictly increasing sequences of integers $n_k,m_k$ by induction. Let $n_1=0$ and if $n_k$ is already defined then let $m_k$ be the smallest integer such that 
\begin{equation} \label{eq:Phi} 
2^{-m_k}< \Phi\left(2^{-kn_k}\right), 
\end{equation}
and let $n_{k+1}=(k+1)m_k$. This implies that $m_1>0$, and defines the sequences $n_k,m_k$. Note that $n_k<m_k<n_{k+1}$ for each $k$ by the definition.

Before giving the detailed construction of our compact set $K\subset [0,1]^n$, we sketch how it looks like. We can assign an infinite tree $T(K)$ to $K$ such that its vertices are the dyadic cubes of $[0,1]^n$ which intersect $K$, and a dyadic cube $Q$ of level $m$ is the parent of another dyadic cube $Q'$ if and only if $Q'$ is of level $m+1$ and $Q'\subset Q$. We can describe $T(K)$ as follows. For each $k$ there will be exactly one vertex in $T(K)$ at level $n_k$ with the maximal, $2^{n(m_k-n_k)}$ many descendants of level $m_k$; let us call them exceptional vertices. This defines the maximal, $2^n$ many children for some vertices; any other vertex in our construction will have just $1$ offspring. Then $T(K)$ is sparse enough to guarantee $\rdim K=0$. We also arrange the exceptional vertices such that every infinite branch of $T(K)$ contains infinitely many of them; this will imply that $\PhiMBdim K=n$. 

Now we turn to the precise construction of our compact set $K\subset \R^n$. 
Let $\ell_1=1$, $q_0=0$, and $C_1=Q^0_1=[0,1]^n$. Assume by induction that $\ell_k,q_{k-1}\geq 0$ and $C_k=\bigcup_{i=1}^{\ell_k} Q^{n_k}_i$ are already defined such that $Q^{n_k}_i$ are different dyadic cubes of level $n_k$. Let
\begin{equation*}
q_{k}=\min\left\{i>q_{k-1}: \text{there exists } j\in \{1,\dots,\ell_k\} \text{ such that } Q^{n_k}_j\subset Q_i\right\}.
\end{equation*}
As $\{Q_i: i>q_{k-1}\}$ contains all the $n_k$th level cubes, $q_k$ is well-defined. Choose an arbitrary $j_k\in \{1,\dots,\ell_k\}$ such that $Q^{n_k}_{j_k}\subset Q_{q_{k}}$. Let
\begin{equation*} 
\ell_{k+1}=\ell_k-1+2^{n(m_k-n_k)}.
\end{equation*}
For all $i\in \{1,\dots,\ell_k\}\setminus \{j_k\}$ choose one $n_{k+1}$th level dyadic subcube of $Q^{n_k}_i$, and for $j_k$ pick all the $2^{n(m_k-n_k)}$ many $m_k$th level dyadic subcubes of $Q^{n_k}_{j_k}$, and choose exactly one $n_{k+1}$th level dyadic subcube of each of them. We obtained $\ell_{k+1}$ dyadic cubes of level $n_{k+1}$, let us enumerate them as $\{Q^{n_{k+1}}_i\}_{1\leq i\leq \ell_{k+1}}$. Since any $m_k$th level cube contains at most $1$ subcube $Q^{n_{k+1}}_i$, using $n_{k+1}=(k+1)m_k$ we obtain 
\begin{equation} \label{ellk}
\ell_{k+1}\leq 2^{m_k}=2^{\frac{n_{k+1}}{k+1}}.
\end{equation}
Let us define 
\begin{equation*}
C_{k+1}=\bigcup_{i=1}^{\ell_{k+1}} Q^{n_{k+1}}_i,
\end{equation*}
clearly $C_{k+1}\subset C_k$ for all $k\geq 1$. Let 
\begin{equation*}
K=\bigcap_{k=1}^{\infty} C_k.
\end{equation*} 

First we prove that $\rdim K=0$. Assume that a scale set $D\in \iD$ is given. Choose a small $r_1\in D$ and large $k$ such that $\sqrt{n} 2^{-n_k} \leq r_1$ and $D\cap [\sqrt{n}2^{-n_{k+1}},\sqrt{n}2^{-n_k}]\neq \emptyset$, and choose $r_2\in D\cap [\sqrt{n}2^{-n_{k+1}},\sqrt{n}2^{-n_k}]$. Then $Q^{n_k}_{j_k}$ has diameter at most $r_1$, and $C_k\setminus Q^{n_k}_{j_k}$ is a union of less than $\ell_k$ dyadic cubes of level $n_{k+1}$. Therefore, they can be covered by $\ell_k$ sets of diameter 
$r_2\leq \sqrt{n}2^{-n_k}$. For any $s>0$ by \eqref{ellk} we obtain 
\begin{equation*}
\iH^s_{D,\infty}(K)\leq r_1^s+\ell_k r_2^s\leq r_1^s+n^{\frac s2}2^{-n_k\left(s-\frac 1k\right)}
\end{equation*} 
which can be arbitrarily small if $r_1$ is small enough and $k\to \infty$. This shows that $\drdim K=0$.

Finally, let $\Phi$ be given, we need to prove that $\PhiMBdim K=n$, then $\pdim K=n$ will simply follow from Claim~\ref{c:ineq}. Assume that $U\subset \RR^n$ is an arbitrarily fixed open set such that $U\cap K\neq \emptyset$. By Lemma~\ref{l:Baire} it is enough to prove that $\PhiBdim(K\cap U)=n$. 
By construction, for every $k\ge 2$ and $1\le i\le \ell_k$ the
cube $Q^{n_k}_i$ is contained in a cube $Q^{n_{k-1}}_{i'}$ and contains a cube $Q^{n_{k+1}}_{i''}$ for some $i'$ and $i''$. 
Using this and the assumption that the sequence of $|Q_i|$ is non-increasing, the definition of $q_k$ yields that 
\begin{equation*} 
\{Q_{q_k}: k\geq 0\}=\{Q_j: \text{ there exist } k\geq 1 \text{ and } 1\leq i\leq \ell_k \text{ such that } Q^{n_k}_i\subset Q_j\}.
\end{equation*} 
Therefore, there exists an infinite set $I\subset \NN$ such that $Q^{n_k}_{j_k}\subset U$ for each $k\in I$. 
For $k\in I$ assume that $N_k=2^{n(m_k-n_k)}$ and $\{P_i\}_{1\leq i\leq N_k}$ is the set of $m_k$th level subcubes of $Q^{n_k}_{j_k}$. By the construction we can choose points $x_i\in P_i\cap K$ for all $1\leq i\leq N_k$. Define the Borel probability measure $\mu_k$ supported within $U\cap K$ as  
\begin{equation*} 
\mu_k=\frac{1}{N_k}\sum_{i=1}^{N_k}   \delta_{x_i},
\end{equation*} 
where $\delta_{x}$ is the Dirac measure concentrated on $\{x\}$.
Let $\eps_k=2^{-kn_k}$ for each $k\in I$, and fix $k\in I$.
By \cite[Lemma~5.1\,(i)]{Ban} it is enough to prove that there is a finite constant $C_n$ such that every Borel set $B\subset K\cap U$ with $\Phi(\eps_k) \leq |B| \leq \eps_k$ satisfies 
\begin{equation} \label{eq:ok1}
\mu_k(B)\leq C_n |B|^{n-o_k(1)}, 
\end{equation}
where $o_k(1)\to 0$ as $k\to \infty$. Let us fix $B$ as above. The definition of $\eps_k$ and \eqref{eq:Phi} imply that $2^{-m_k}<|B|\leq 2^{-kn_k}$, so there exists a unique integer $p$ such that $kn_k+1\leq p\leq m_k$ and 
\begin{equation} \label{eq:p} 
2^{-p}<|B|\leq 2^{1-p}.
\end{equation} 
Clearly, there is a finite constant $C_n$ such that $B$ can cover at most $C_n 2^{n(m_k-p)}$ points $x_i$. 
Therefore, using $n_k\leq \frac pk$ and \eqref{eq:p} we obtain that  
\begin{equation*} \mu_k(B)\leq \frac{1}{N_k}  C_n 2^{n(m_k-p)}=C_n 2^{n(n_k-p)}\leq C_n 2^{-pn\left(1-\frac 1k\right)}\leq C_n |B|^{n\left(1-\frac 1k\right)}.
\end{equation*}
Hence \eqref{eq:ok1} holds, which implies that $\PhiMBdim K=n$. The proof of the theorem is complete. 
\end{proof}

It is natural to ask whether $\sup$ can be replaced by $\max$ in the definition of $\rdim$. The next proposition shows that there are compact counterexamples of the form \eqref{d:ASi}.

\begin{prop} \label{p:rdim}
There exists a compact set $K\su\R$ of the form \eqref{d:ASi} for which $\drdim K<\rdim K$ for all $D\in\iD$.
\end{prop}

\begin{proof}
Let $K=A(\{S_i: i\geq 1 \})$ where the sets $S_i$ are defined as follows. Let $p_i$ be the $i$th odd prime number, and for all $i\geq 1$ define 
\begin{equation*} 
S_i=\bigcup_{j=1}^{\infty} \left\{(p_i^j)!,\dots, i\big((p_i^j)!\big)\right\} \quad \textrm{and} \quad 
P_i=\bigcup_{j=1}^{\infty} \left\{(p_i^j)!,\dots,(p_i^j+1)!-1\right\}.
\end{equation*}
It is easy to see that $\limsup_k M_{S_i}(k)=\frac{i-1}{i}$. Therefore the $\sigma$-stability of $\rdim$, Theorems~\ref{t:ASDarboux} and \eqref{e:pdimAS} imply
\begin{equation} \label{e:rdK} 
\rdim K=\sup_{i\geq 1} \rdim A_{S_i}
=\sup_{i\geq 1} \pdim A_{S_i} 
=\sup_{i\geq 1} \limsup_{k\to \infty} M_{S_i}(k)=1.
\end{equation}
Fix $D\in \iD$. By Remark~\ref{r:wlogD} we may again assume that
\begin{equation*}
D=\bigcup_{k=1}^\infty [2^{-n_k},2^{-n_k+1}),
\end{equation*}
where $n_k$ is a strictly increasing sequence of positive integers. Clearly $S_i\subset P_i$ for every $i$, and $P_i\cap P_j=\emptyset$ whenever $i\neq j$. Thus, there is at most one index $i$ such that $n_k\in P_i$ for all but finitely many numbers $k$; let it be $i_0$ if there exists such an index, otherwise fix an arbitrary $i_0$. It is easy to see that for all $i\in \N\setminus\{i_0\}$ we have 
\begin{equation} \label{e:densM} 
\liminf_{k\to \infty} M_{S_i}(n_k)=0 \quad \textrm{and} \quad  \liminf_{k\to \infty} M_{S_{i_0}}(n_k)\leq \frac{i_0-1}{i_0}<1.
\end{equation} 
Then \eqref{e:dimASi}, \eqref{e:densM}, and \eqref{e:rdK} imply that 
\begin{equation*} 
\drdim K=\sup_{i\geq 1} \liminf_{k\to\infty} M_{S_i}(n_k)\leq \frac{i_0-1}{i_0}<1=\rdim K,
\end{equation*}
and the proof is complete.
\end{proof}

The following lemma is surely known but for the sake of completeness we give a short proof.

\begin{lemma}\label{l:balls}
Let $0<\eps\le\delta/2$. In $\R^n$ every open ball of diameter $\delta$ contains  $\left\lceil\left(\frac{\delta}{4\eps}\right)^n\right\rceil$ pairwise disjoint closed balls of diameter $\eps$.
\end{lemma}

\begin{proof}
Suppose that we have $m$ disjoint closed balls of diameter $\eps$ inside an open ball of diameter $\delta$ and this $m$ is maximal. 
Then $m$ closed balls of diameter $2\eps$ covers an open ball of diameter $\delta-\eps>\delta/2$, so $(\delta/2)^n\le m(2\eps)^n$, which implies that $m\ge \left(\frac{\delta}{4\eps}\right)^n$. 
\end{proof}

The following theorems explain why the list of properties in Proposition~\ref{p:easyrdim} is shorter than in Proposition~\ref{p:easydrdim}: properties \eqref{pD:5} and \eqref{pD:6} of Proposition~\ref{p:easydrdim} do not hold for $\rdim$.
Recall Subsection~\ref{ss:Baire} for the definition of meager sets and other Baire category notions.

\begin{theorem} Every set $E\subset \R^n$ which is co-meager in a non-empty open set contains a compact set $K$ with $\rdim K=n$. 
\end{theorem}

\begin{proof} 
By assumption there is an open ball $D_\emptyset\su\R^n$ and there are nowhere dense sets $A_1, A_2, \ldots$ such that
\begin{equation}\label{e:DminE}
    D_\emptyset\setminus E=\bigcup_{k=1}^\infty A_k.
\end{equation}

We will construct a compact set $K\subset D_{\emptyset}\cap E$ implementing the following plan. Once we have defined balls $B_{i_1\dots i_k}$ of the same diameter disjoint from $A_1\cup\dots \cup A_k$, then we can find smaller balls $D_{i_1\dots i_k}\subset B_{i_1\dots i_k}$ which are disjoint from $A_{k+1}$ as well. We choose a large number of small balls $B_{i_1\dots i_{k+1}}\subset D_{i_1\dots i_k}$ of the same diameter, and repeat this process. Finally, the Cantor scheme $K=\bigcap_k\left(\bigcup_{i_1}\dots \bigcup_{i_k} B_{i_1\dots i_k}\right)$ gives our example, which is clearly disjoint from $\bigcup_{k=1}^\infty A_k$.

More precisely, let $\delta_1$ be the diameter of $D_\emptyset$. 
By induction we can construct sequences of positive numbers $\{\eps_i\}_{i\geq 1}$ and $\{\delta_i\}_{i\geq 1}$, positive integers $\{m_i\}_{i\geq 1}$, open balls $B_{i_1 \dots i_k}$ and $D_{i_1\ldots i_k}$ $(1\leq i_1\leq m_1,\dots,1\leq i_k\leq m_k)$ such that for each $k\geq 1$ we have 
\begin{enumerate}
\item \label{e:ballsB} the balls $B_{i_1 \dots i_k}$ $(1\leq i_1\leq m_1,\dots,1\leq i_k\leq m_k)$ have pairwise disjoint closures and diameter $\eps_k$,
\item the balls $D_{i_1 \dots i_{k-1}}$ $(1\leq i_1\leq m_1,\dots,1\leq i_{k-1}\leq m_{k-1})$ have diameter $\delta_k$,
\item \label{e:DBD}
$D_{i_1 \dots i_{k}}\subset B_{i_1 \dots i_{k}}\subset D_{i_1 \dots i_{k-1}}$ $(1\leq i_1\leq m_1,\dots,1\leq i_k\leq m_k)$,
\item \label{e:disjointoA}
$D_{i_1 \dots i_k}\cap A_k=\emptyset$ $(1\leq i_1\leq m_1,\dots,1\leq i_k\leq m_k)$,
\item \label{e:epsdelta} 
$\eps_k\le \delta_k/2$, $\eps_k\leq \delta_k^k$ and $m_k = \left\lceil\left(\frac{\delta_k}{4\eps_k}\right)^n\right\rceil$.
\end{enumerate}
Let $k\in\N$ and suppose that $D_{i_1\ldots i_{k-1}}$ has been defined and has diameter $\delta_k$, which holds for $k=1$.
Choose $\eps_k$ and $m_k$ according to \eqref{e:epsdelta}.
By Lemma~\ref{l:balls} in each $D_{i_1\ldots i_{k-1}}$ we can choose $m_k$ open balls $B_{i_1\ldots i_{k}}$ ($i_k=1,\ldots,m_k$) with disjoint closure and of diameter $\eps_k$.
Since $A_k$ is nowhere dense, in each ball $B_{i_1\ldots i_{k}}$ we can find an open ball $D_{i_1\ldots i_{k}}$ disjoint to $A_k$. By shrinking some of them we can make their size to be the same. Let $\delta_{k+1}$ be their common diameter. This completes the construction with the above properties.

Let 
\begin{equation*}K=\bigcap_{k=1}^{\infty} 
\left(\bigcup_{i_1=1}^{m_1} \dots \bigcup_{i_k=1}^{m_k} B_{i_1\dots i_k} \right).
\end{equation*}
By \eqref{e:DminE}, \eqref{e:DBD} and \eqref{e:disjointoA} we have $K\su E$.

Let $D=\cup_{k=1}^{\infty} [\eps_k, 2\eps_k)$. It is enough to prove that $\drdim K=n$. Let $\mu$ be the natural probability measure supported on $K$ such that for all $k\geq 1$ and $i_1\in \{1,\dots,m_1\}, \dots, i_k \in \{1,\dots,m_k\}$ we have  
\begin{equation*} 
\mu(B_{i_1\dots i_k})=\frac{1}{m_1\cdots m_k}.
\end{equation*} 
Let $U\su\R^n$ be a closed set with $|U|\in D$.
Then for some $k$ we have $|U|\in[\eps_k, 2\eps_k)$ and there is a ball of radius $2\eps_k$ that contains $U$.
Then by \eqref{e:ballsB}, $B$ can intersect at most $C_n$ sets of the form $B_{i_1\dots i_k}$ with some finite constant $C_n$, so \eqref{e:epsdelta} yields
\begin{equation*} 
\mu(U)\le\mu(B)\leq \frac{C_n}{m_1\cdots m_k}\leq \frac{C_n}{m_k}\leq C_n \left(\frac{4\eps_k}{\delta_k}\right)^n\leq C_n 4^n \eps_k ^{n(1-\frac1k)} 
\le C_n 4^n |U|^{n(1-\frac1k)}. 
\end{equation*} 
Applying the mass distribution principle for the $D$-diameter restricted Hausdorff dimension (Lemma~\ref{l:mass}) implies that $\drdim K\ge n(1-\frac1k)$ for any $k\ge 1$, and the proof is complete.
\end{proof}

\begin{cor}
Each non-meager set $E\su\R^n$ with the Baire property satisfies $\rdim E=n$, and no dense set can be covered by a $G_\delta$-set of universal restricted Hausdorff dimension zero.
\end{cor}

Let us denote by $\iK(\R^n)$ the set of non-empty 
compact subsets of $\R^n$ endowed with the Hausdorff metric.

\begin{theorem} The function $\rdim \colon \iK(\R^n)\to [0,n]$ is not Borel for each $n\geq 1$.
\end{theorem}

\begin{proof}
We follow the proof of \cite[Theorem~7.5]{MM}. 
Let $\II=[0,1]\setminus \QQ$ be the set of irrational numbers in $[0,1]$.
For each $z\in \II$ there is a unique infinite continued fraction expansion as  
\begin{equation*}
[b_1,b_2,\dots]=\frac{1}{b_1+\frac{1}{b_2+\dots}}, 
\end{equation*}
where $b_i\in \NN$. 
For $z=[a_1,a_2,\dots]\in \II$ define 
\begin{equation*}
N(z)=\{[a_1,k_1,a_2,k_2,\dots]: k_i\in \N,\, i\geq 1 \}.
\end{equation*}
\begin{claim} $\rdim N(z)\geq 1/2$ for all $z\in \II$. 
\end{claim}
\begin{proof}[Proof of the claim]
Fix $z=[a_1,a_2,\dots]\in \II$, we will define large enough positive integers $n_i\geq 11$ depending on the sequence $a_i$, and consider the compact set $M\subset N(z)$ defined as 
\begin{equation*}
M=\{[a_1,k_1, a_2, k_2\dots]: 1\leq k_i\leq n_i,~k_i\in \NN,~i\geq 1 \}.
\end{equation*} 
For a given $m\in\N$ let $J(k_1,\dots, k_m)$ be the smallest interval containing the points of $M$ whose expansions start with $k_1,\dots,k_m$. 
It is proved in \cite[Pages~91--92]{MM} that the $m$th level intervals $J(k_1,\dots, k_m)$ are pairwise disjoint, and their lengths are at least 
\begin{equation*}
\delta_m:=\frac{1}{4^{2m+2}(a_1\cdots a_{m+1})^2(n_1\cdots n_m)^2}.
\end{equation*}
Fix the sequence $n_m\geq 11$ to be so large that 
\begin{equation} \label{e:asympt}
\lim_{m\to \infty} \frac{\log\delta_m}{\log(n_1\cdots n_m)}=-2.
\end{equation}
Define  
\begin{equation*} 
D=\bigcup_{m=1}^{\infty} [\delta_m,2\delta_m),
\end{equation*}
it is enough to prove that $\drdim M\geq 1/2$. 

As every $m$th level interval $J(k_1,\dots, k_m)$ contains exactly $n_{m+1}$ intervals of level $m+1$, we can uniquely define a Borel probability measure $\mu$ supported on $M$ by setting 
\begin{equation*}
\mu(J(k_1,\dots, k_m))=\frac{1}{n_1\cdots n_m}.
\end{equation*} 
Fix an arbitrary $\eps>0$. Every interval $I$ of length $d\in [\delta_m,2\delta_m)$ can intersect at most three 
$m$th level intervals $J$, so by \eqref{e:asympt} for all large enough $m$ we have
\begin{equation*}
 \mu(I)\leq \frac{3}{n_1\cdots n_m}\leq (\delta_m)^{\frac 12-\eps}\leq d^{\frac 12-\eps}.
\end{equation*}
The mass distribution principle for the $D$-diameter restricted Hausdorff dimension (Lemma~\ref{l:mass}) yields $\drdim M\geq 1/2-\eps$. Since $\eps>0$ was arbitrary, we obtain that $\drdim M\geq 1/2$, and the proof of the claim is complete.
\end{proof}
Now copying the proof of \cite[Theorem~7.5\,(b)]{MM} concludes that the set 
\begin{equation*}
\iK_0=\left\{K\in \iK\left([0,1]\times \{0\}^{n-1}\right): \rdim K>0\right\} 
\end{equation*}
is a non-Borel set. Since $\iK_0$ is the intersection of 
\begin{equation*} \iK_1=\{K\in \iK(\RR^n): \rdim K>0\}
\end{equation*} 
and a closed subset of $\iK(\R^n)$, 
the set $\iK_1$ cannot be Borel, which completes the proof of the theorem.
\end{proof}

\section{Bilipschitz invariant nice dimensions}
\label{s:bilipdim}

The authors proved in \cite{BK} that among those Lipschitz invariant, monotone dimensions 
on the compact subsets of $\R^n$ that agree with the similarity dimension for homogeneous SSC self-similar sets, the Hausdorff dimension is the smallest, the upper box dimension is the greatest. Alex Rutar asked what happens if Lipschitz invariance is replaced by bilipschitz invariance. 

\subsection{The smallest bilipschitz invariant dimensions} The next result shows that the smallest such dimension is the modified lower dimension $\mldim$.

\begin{theorem}\label{t:mldim}
Let $n$ be a positive integer and let $D$ be a function from the family of compact subsets of $\R^n$ to $[0,n]$. Suppose that 
\begin{itemize}
\item[$(\star)$] \parbox{\linewidth}{$D$ is bilipschitz invariant, monotone, and it agrees with the similarity dimension for any 
homogeneous self-similar set with the SSC.}
\end{itemize}
 
Then for every compact set $K\su \R^n$ we have \begin{equation} \label{e:dimMLatmostD}
\mldim(K)\le D(K).
\end{equation}
\end{theorem}

\begin{proof} 
Assume to the contrary that $K\su \R^n$ is compact and $D(K)<\mldim K$. Choose $s$ and $t$ such that  $D(K)<t<s<\mldim K$.
Then the proof of \cite[Theorem~3.4.3]{F} yields that any compact set $K\su\R^n$ with $\mldim K>s$ contains an $s$-Ahlfors--David regular compact subset $C$.
Take such a $C$. By a theorem of Mattila and Saaranen \cite[3.1.~Theorem]{MS}, if
$0<t<s$ then for any $s$-regular set $C\su\R^n$ there exists a $t$-dimensional homogeneous self-similar set $S\subset \R^n$ with the SSC which is bilipschitz equivalent to a subset $F$ of $C$.
Then $F$ is clearly compact, and by $(\star)$ we obtain
\begin{equation}
t=\dim(S)=D(S)=D(F)\le D(C)\le D(K)<t,
\end{equation}
which is a contradiction.
\end{proof}

\begin{remark}
We claim that in Theorem~\ref{t:mldim} we cannot replace the family of compact sets by a family $\iF$ that contains all $\sigma$-compact sets of $\R^n$ so that \eqref{e:dimMLatmostD} still holds for every $A\in\iF$. 
Indeed, for any $A\in\iF$ let
\begin{align*} \label{e:mindim}
D_1(A)=\sup \{&\dim S\colon  S \textrm{ is a homogeneous self-similar set with the SSC,}  \\ 
& \textrm{which is bilipschitz equivalent to a subset of } A\},
\end{align*}
where we use the convention $\sup\emptyset=0$.
Clearly $D_1$ satisfies $(\star)$ on $\iF$ as well, and $D_1(\Q)=0$. Therefore, we obtain $D_1(\Q)=0<1=\mldim \Q$.

In fact, it is easy to see that if $\iF$ is a collection of subsets of $\R^n$ containing all compact sets of $\R^n$ and $D\colon \iF\to[0,n]$ satisfies $(\star)$, then we have $D(A)\ge D_1(A)$ for any $A\in\iF$. This implies that on a general domain $D_1$ is the smallest dimension that satisfies $(\star)$.
\end{remark}

\subsection{The largest bilipschitz invariant dimensions}
Alex Rutar asked whether the Assouad dimension is the greatest dimension on compact sets for which $(\star)$ holds. Before proving an affirmative answer in $\R$ and some positive results in higher dimensions, we give a simple example showing that in this form the answer is negative in $\R^n$ for every $n>1$.

\begin{example} \label{ex:Assouad}
Let $n>1$. For any compact set $K\su\R^n$ let
\begin{equation*}
    D_2(K) = \begin{cases}
  \adim K  & \text{ if } $K$ \text{ is totally disconnected}, \\
  n & \text{ otherwise}.
\end{cases}
\end{equation*}
Then $D_2$ is a function from the family of compact subsets of $\R^n$ to $[0,n]$ satisfying $(\star)$, and for any compact line segment $L$ we have $D_2(L)=n>1=\adim L$.
\end{example}

The following example gives the greatest dimension on compact sets for which $(\star)$ holds, and it also shows that $\adim$ is not an upper bound even for totally disconnected Cantor sets when $n>1$.

\begin{example}
 For any compact set $K\su\R^n$ let
\begin{equation*}
    D_3(K) = \begin{cases}
  \inf \{ \dim S : S \in \iS(K) \}  & \text{ if } \iS(K)\neq\emptyset, \\
  n & \text{ otherwise},
\end{cases}
\end{equation*}
where $\iS(K)$ is the set of those homogeneous self-similar sets with the SSC that contain a subset bilipschitz equivalent to $K$.

Then $D_3$ is a function from the family of compact subsets of $\R^n$ to $[0,n]$ satisfying $(\star)$. We claim that $D(K)\le D_3(K)$ holds for any such $D$. Indeed, assume to the contrary that $D(K)>D_3(K)$ for a compact set $K$, then there exist sets $S\in\iS(K)$ and $S'\su S$ such that $\dim S<D(K)$ and $S'$ is bilipschitz equivalent to $K$. 
Then $(\star)$ implies 
\begin{equation*}
 D(K)= D(S')\leq D(S)=\dim S<D(K),
\end{equation*}
which is a contradiction.

Finally, we claim that $D_3$ can be greater than $\adim K$ even for totally disconnected compact sets when $n>1$.
Indeed, it is proved in \cite{Mo} (and in fact not hard to show directly) that in $\R$ a Cantor set of positive measure and the compact set $\{0, 1, 1/2, 1/3,\ldots\}$ are not bilipschitz equivalent to any ultrametric space.
So let $n>2$ and $K$ be one of the above subsets of a line in $\R^n$.
Then $K$ is a totally disconnected compact set and $D_3(K)=n>1\ge \adim K$.
\end{example}

A `reasonable' dimension should be $k$ for any $k$-dimensional compact smooth surface, so it is natural to add this requirement to $(\star)$ and one might hope an affirmative answer to Alex Rutar's question with this modification. The following example shows that this is not the case.

\begin{example}
Let $n>1$. For any compact $K\su\R^n$ let
\begin{equation*}
D_4(K) = 
\begin{cases}
  \adim K  & \text{ if } K \text{ is totally disconnected or } \adim K\le n-1, \\
  n & \text{ otherwise}.
\end{cases}
\end{equation*}
Again, it is easy to check that $D_4$ is a function from the family of compact subsets of $\R^n$ to $[0,n]$, it satisfies $(\star)$ and $D_4(F)=k$ for any $k$-dimensional compact smooth surface $F$ for any $k=1,\ldots,n-1$. On the other hand, if $E$ is a compact set with $\adim E=n-1/2$ and $F$ is a smooth compact surface then
\begin{equation*} 
D_4(E\cup F)=n>n-1/2=\adim(E\cup F).
\end{equation*}
\end{example}

In order to prove positive results, first we need some preparation. 
\begin{notation}
Let $\iR\subset X\times Y$ be a relation. Its \emph{domain} is defined as
\begin{equation*} \dom \iR=\{x\in X: \exists y\in Y \text{ such that } (x,y)\in \iR\}.
\end{equation*} 
The \emph{image} of a set $A\subset X$ under the relation $\iR$ is 
\begin{equation*} \iR(A)=\{y\in Y: \exists x\in A \text{ such that } (x,y)\in \iR\}.
\end{equation*} 
\end{notation}
K\"aenm\"aki and Rutar \cite[Definition~2.1]{KR} generalized Lipschitz maps as follows.

\begin{defi} \label{d:Ldec}
Let $(X,d_X)$ and $(Y,d_Y)$ be metric spaces and let $\iR\subset X\times Y$ be a relation. We say that $\iR$ is \emph{Lipschitz decomposable} if $\dom \iR=X$ and there are constants $M\in \NN$ and $c>0$ such that for all $x\in X$ and $r>0$ there are $y_1,\dots,y_M\in Y$ such that 
\begin{equation*}
\iR(B_X(x,r))\subset \bigcup_{i=1}^M B_Y(y_i,cr).
\end{equation*}
Note that we can take $M=1$ if and only if $\iR$ is a Lipschitz map.   
The relation $\iR$ is said to be \emph{bilipschitz decomposable} if both $\iR$ and $\iR^{-1}$ are Lipschitz decomposable. For a map $f\colon X\to Y$ we identify $f^{-1}$ with the relation $f^{-1}=\{(y,x): f(x)=y\}$. 
\end{defi}

K\"aenm\"aki and Rutar pointed out in \cite[Lemma~2.3]{KR} that the Assouad dimension is not only bilipschitz invariant, but invariant under the more general bilipschitz decomposable relations as well; this is the first part of the next lemma. The second part of the following lemma easily follows from the definitions of $s$-regular sets and the Hausdorff dimension.

\begin{lemma} \label{l:KR} Let $X,Y\subset \RR^n$ and $\iR\subset X\times Y$ be a bilipschitz decomposable relation.
\begin{enumerate}[(i)]
\item \label{i:dec1} Then $\adim X=\adim Y$; 
\item \label{i:dec2} if $X$ is $s$-regular, then $Y$ is also $s$-regular.
\end{enumerate}
\end{lemma}

The construction of the following lemma is quite standard.

\begin{lemma} \label{l:lipimage} Let $K\su\R^n$ be bounded with $\adim K < t\leq n$. Then there exist a homogeneous self-similar set $C\subset \RR^n$ with the SSC and $\dim C\in (\adim K, t)$, and a Lipschitz map $f\colon C\to \RR^n$ such that $f^{-1}$ is Lipschitz decomposable and $K\subset f(C)$.
\end{lemma}

\begin{proof}
First we cover $K$ by a compact set $L$ obtained by a regular tree construction, meaning that each vertex at every level has the same number of offspring. Then we will construct a self-similar $C$ by an identical tree construction but with disjoint pieces repeating the same pattern, and $f\colon C\to L$ will be the natural surjection given by the identical self-map of our tree.

As $\adim  K<t$, we can choose positive integers $k_0,\ell_0$ such that $s:=\frac{\log_2 k_0}{\ell_0}$
satisfies 
$s\in(\adim K, t)$.
Let $v$ be a large enough positive integer that we will choose later and let $k=k_0^v$ and $\ell = \ell_0 v$, so $\frac{\log_2 k}{\ell}=\frac{\log_2 k_0}{\ell_0}=s$. 
Recall that dyadic cubes of level $m$ are cubes of the form 
$[a_1 2^{-m}, (a_1+1) 2^{-m}]\times\ldots\times[a_n 2^{-m}, (a_n+1) 2^{-m}]$, where $a_1,\ldots,a_n\in\Z$.
Since $\adim K<s$ and $k=2^{\ell s}$, for every dyadic cube $Q$ of any level $m$, the set $K\cap Q$ can be covered by at most $k$ dyadic cubes of level $m+\ell=m+\ell_0 v$ if we choose $v$ large enough. 
Since $s<t\le n$, by choosing $v$ large enough, we can also guarantee that $\ell$ is large enough to have $k=2^{\ell s}\le 2^{(\ell-1)n}$.

We may assume that $K\subset [0,1]^n$. Let $Q_{\emptyset}=[0,1]^n$ and if the dyadic cube $Q_{i_1\dots i_m}$ of level $m\ell$ is defined for some $m\geq 0$ then for $1\leq i\leq k$ let $Q_{i_1\dots i_m,i}\su Q_{i_1\dots i_m}$ be different dyadic cubes of level $(m+1)\ell$ such that
\begin{equation*}
K\cap Q_{i_1\dots i_m} \subset \bigcup_{i=1}^k Q_{i_1\dots i_m,i}.
\end{equation*} 
This way we obtain a `uniformly branching set' 
\begin{equation*}
L=\bigcap_{m=0}^{\infty} 
\left( \bigcup_{i_1=1}^k\dots  \bigcup_{i_m=1}^k Q_{i_1\dots i_m} \right)
\end{equation*} 
such that $K\subset L$.

Now we define a homogeneous self-similar set $C$ with the SSC and with the same structure as $L$ and a natural onto map $f\colon C\to L$ with the required properties.
Let $Q'_{\emptyset}=Q_{\emptyset}=[0,1]^n$. Since $k\le 2^{(\ell-1)n}$ we can choose $k$ disjoint (closed) dyadic cubes $Q'_1,\ldots,Q'_\ell\subset Q'_{\emptyset}$ of level $\ell$. 
By repeating the same pattern we obtain cubes $Q'_{i_1\ldots i_m}$ and we get a homogeneous self-similar set 
\begin{equation*} C=\bigcap_{m=0}^{\infty} 
\left( \bigcup_{i_1=1}^k\dots  \bigcup_{i_m=1}^k Q'_{i_1\dots i_m} \right)
\end{equation*} 
with the SSC such that $\dim C=s$.
Let $f\colon C\to L$ be the natural function such that
$f(C\cap Q'_{i_1\ldots i_m})=L\cap Q_{i_1\ldots i_m}$
for every positive integer $m$ and $i_1,\dots,i_m\in \{1,\dots,\ell\}$.
It is easy to see that $f$ is Lipschitz and $f(C)=L$. 

Finally, we need to prove that $f^{-1}$ is Lipschitz decomposable. One can easily check that in $\R^n$ the balls can be replaced by dyadic cubes in Definition~\ref{d:Ldec}. Since the preimage of any $m$th level dyadic cube $Q_{i_1\ldots i_m}$ is covered by at most $3^n$ many $m$th level dyadic cubes, this completes the proof of the lemma.
\end{proof}

Lemma~\ref{l:lipimage} and Lemma~\ref{l:KR}~\eqref{i:dec2} immediately imply the following.  

\begin{cor} \label{c:regcover}
If $K\su\R^n$ is a bounded set and $\adim K < t\leq n$ then there exist an $s\in(\adim K, t)$ and an $s$-regular compact set $L\su\R^n$ such that $K\subset L$.
\end{cor}

We also need the following lemma. 

\begin{lemma}\label{l:sscAssouad}
If $K\su\R^n$ is a bounded set with $\adim K<t<1$ then there exists a homogeneous self-similar set $S\su\R^n$ with the SSC such that $\adim S=t$ and $K$ is bilipschitz equivalent to a subset of $S$.
\end{lemma}

\begin{proof}
By Corollary~\ref{c:regcover} there exist an $s\in(\adim K, t)$ and an $s$-regular compact $L\su\R^n$ such that $K\su L$. 
Let $S\su\R^n$ be a homogeneous self-similar set with the SSC having dimension $t$.
By a theorem of Mattila and Saaranen \cite[3.3.~Theorem]{MS} if $E,F\su\R^n$ are compact sets such that $E$ is $s$-regular, $F$ is $t$-regular such that $0<s<t<1$, then $E$ is bilipschitz equivalent to a subset of $F$. 
Applying this to $E=L$ and $F=S$ 
completes the proof.
\end{proof}

\begin{theorem}\label{t:adim}
For a fixed positive integer $n$ let $\iF$ be any family of subsets of $\R^n$ that contains all compact subsets of $\R^n$.
Let $D\colon \iF\to[0,n]$ be such that
\begin{itemize}
\item[$(\star)$] \parbox{\linewidth}{$D$ is bilipschitz invariant on the compact sets, monotone, and it agrees with the similarity dimension for any 
homogeneous self-similar set with the SSC.}
\end{itemize}
Then for every bounded set $K\in\iF$ with $\adim K<1$ we have 
\begin{equation}\label{e:DatmostdimA} 
D(K)\le \adim K.
\end{equation}
In particular, if $n=1$ then \eqref{e:DatmostdimA} holds for every bounded set $K\in\iF$.
\end{theorem}

\begin{proof}
Assume to the contrary that there exists a bounded $K\in\iF$ such that $\adim K < \min\{D(K),1\}$. 
As replacing $K$ by its closure cannot increase $\adim K$ and decrease $\min\{D(K),1\}$, we may suppose that $K$ is compact.
Choose $t$ such that  $\adim K < t < \min(D(K),1)$.
Lemma~\ref{l:sscAssouad} yields a $t$-dimensional homogeneous self-similar set $S\su\R^n$ with the SSC and a set $S'\subset S$ which is bilipschitz equivalent to $K$.
Then $S$ and $S'$ are clearly compact, so $S, S'\in\iF$ and by $(\star)$ we obtain \begin{equation*} t=\dim(S)=D(S)\ge D(S') = D(K)>t,
\end{equation*} 
which is a contradiction. 
\end{proof}

We saw in Example~\ref{ex:Assouad} that \eqref{e:DatmostdimA} of Theorem~\ref{t:adim} does not hold for all compact sets when $n>1$. There are two ways to get rid of this obstacle: 
\begin{enumerate}[(A)]
\item \label{i:plan1} we prescribe $D$ on a larger family of sets;
\item \label{i:plan2} we require the invariance of $D$ for a larger family of maps.
\end{enumerate} 

First we consider plan \eqref{i:plan1}. It is well known~\cite{F} that for any compact $s$-regular set $K$ we have $\ldim K=\mldim K=\hdim K=\ubdim K=\adim K=s$, so it is a natural expectation about a reasonable dimension that it should be $s$ for compact $s$-regular sets. 
We show that with this stronger requirement Theorem~\ref{t:adim} also holds in higher dimensions and we do not even need bilipschitz invariance.

\begin{theorem}\label{t:adimregular}
For a fixed positive integer $n$ let $\iF$ be any family of subsets of $\R^n$ that contains all compact subsets of $\R^n$.
Let $D\colon \iF\to[0,n]$ be such that
\begin{itemize}
    \item[$(\star\star)$] $D$ is 
    monotone, and $D(K)=s$ for every compact $s$-regular set $K\su\R^n$.
\end{itemize}
Then for every bounded $K\in\iF$ we have \begin{equation*} D(K)\le \adim K.
\end{equation*}
\end{theorem}

\begin{proof}
Assume to the contrary that $\adim K < D(K)$ for a bounded set $K\in\iF$. By Corollary~\ref{c:regcover} there exist $s\in(\adim K,D(K))$ and an $s$-regular compact set $L\su\R^n$ such that $K\su L$. Then by $(\star\star)$ we obtain $D(K)\leq  D(L)=s<D(K)$, which is a contradiction.
\end{proof}

Finally, we carry out plan \eqref{i:plan2}. 
The next theorem was suggested by Alex Rutar. Note that Lemma~\ref{l:KR}~\eqref{i:dec1} guarantees that the Assouad dimension satisfies property \eqref{i:A1} below, so it is the greatest dimension fulfilling the conditions of the following theorem. 
Note also that this property \eqref{i:A1} is stronger than bilipschitz invariance but weaker than Lipschitz stability.

\begin{theorem} \label{t:decomp}
For a fixed positive integer $n$ let $\iF$ be any family of subsets of $\R^n$ that contains all compact subsets of $\R^n$.
Let $D\colon \iF\to[0,n]$ be such that
\begin{enumerate}[(1)] 
\item \label{i:A1}
$D(f(K))\leq D(K)$ for any compact set $K$ and Lipschitz map $f\colon K\to \RR^n$ for which $f^{-1}$ is Lipschitz decomposable, 
\item \label{i:A2} $D$ is monotone, 
\item \label{i:A3} $D$ agrees with the similarity dimension for every homogeneous self-similar set with the SSC.
\end{enumerate}
Then for every bounded $K\in\iF$ we have \begin{equation*} D(K)\le \adim K.
\end{equation*}
\end{theorem}

\begin{proof}
Assume to the contrary that $\adim K < D(K)$ for a bounded set $K\in\iF$. By Lemma~\ref{l:lipimage} there exist a homogeneous self-similar set $C\subset \RR^n$ with the SSC and $\dim C\in (\adim K, D(K))$, and a Lipschitz map $f\colon C\to \RR^n$ such that $f^{-1}$ is Lipschitz decomposable and $K\subset f(C)$. Applying \eqref{i:A2}, \eqref{i:A1}, and \eqref{i:A3} in this order implies $D(K)\leq  D(f(C))\leq D(C)=\dim C<D(K)$, which is a contradiction.
\end{proof}

\subsection*{Acknowledgments} We are indebted to Amlan Banaji, Jonathan M.~Fraser, and Alex Rutar for some illuminating conversations and helpful suggestions. In particular, Alex Rutar suggested Theorem~\ref{t:decomp}. We also thank the anonymous referees for carefully reading the manuscript; their questions encouraged us to find Theorem~\ref{t:lpacking} and to include Theorem~\ref{t:decomp}.


\begin{thebibliography}{99}





\bibitem{Assouad} P.~Assouad, Espaces m\'etriques, plongements, facteurs, \textit{Th\`ese de doctorat d'\'Etat, Publ.~Math. Orsay} no.~223-7769, Univ.~Paris XI, Orsay (1977).

\bibitem{BK} R.~Balka, T.~Keleti, Lipschitz images and dimensions, \textit{Adv.~Math.}~\textbf{446} (2024), 109669.

\bibitem{BEK} R.~Balka, M.~Elekes, V.~Kiss, Stability and measurability of the modified lower dimension, \textit{Proc.~Amer.~Math.~Soc.}~\textbf{150} (2022), 3889--3898.

\bibitem{Ban} A.~Banaji, Generalised intermediate dimensions, \textit{Monatsh.~Math.}~\textbf{202} (2023), 465--506. 

\bibitem{BP} C.~J.~Bishop, Y.~Peres,
\textit{Fractals in Probability and Analysis},
Cambridge studies in advanced mathematics 162, Cambridge University Press, 2017.

\bibitem{BP2} C.~J.~Bishop, Y.~Peres, Packing dimension and Cartesian products, \textit{Trans.~Amer.~Math. Soc.}~\textbf{348} (1996), no.~11, 4433--4445.

\bibitem{Bou} M.~G.~Bouligand, Ensembles Impropres et Nombre Dimensionnel, \textit{Bull.~Sci.~Math.}~\textbf{52} (1928), 361--376.


\bibitem{DS} Z.~Douzi, B.~Selmi,
Projection Theorems for Hewitt-Stromberg and Modified Intermediate Dimensions, \textit{Results Math}~\textbf{77}, 158 (2022). 

\bibitem{Fa} K.~J.~Falconer, \textit{Fractal geometry: Mathematical foundations and applications}, Second Edition, John Wiley \& Sons, 2003.

\bibitem{FFK}
K.~J.~Falconer, J.~M.~Fraser, T.~Kempton, Intermediate dimensions,
\textit{Math.~Z.}~\textbf{296} (2020), 813--830. 

\bibitem{F} J.~M.~Fraser, \textit{Assouad Dimension and Fractal Geometry}, Cambridge University Press, 2020.

\bibitem{FY} J.~M.~Fraser, H.~Yu, New dimension spectra: finer information on scaling and homogeneity, \textit{Adv.~Math.}~\textbf{329} (2018), 273--328.

\bibitem{KR} A.~K\"aenm\"aki, A.~Rutar, Regularity of non-autonomous self-similar sets, arXiv:2410.17944.

\bibitem{Kec} A.~S.~Kechris, 
\textit{Classical Descriptive Set Theory}, 
Graduate Texts in Mathematics 156, Springer-Verlag, 1994.

\bibitem{L} D.~G.~Larman, A new theory of dimension, \textit{Proc.~London Math.~Soc.}~(3)~\textbf{17} (1967), 178--192.
 
\bibitem{Ma} P.~Mattila, \textit{Geometry of sets and measures in Euclidean spaces}, Cambridge Studies in Advanced Mathematics No.~44, Cambridge University Press, 1995.

\bibitem{MM} P.~Mattila, R.~Mauldin,
Measure and dimension functions: Measurability and densities,
\textit{Math.~Proc.~Cambridge Philos.~Soc.}
\textbf{121} (1997), 81--100. 

\bibitem{MS} P.~Mattila, P.~Saaranen,
Ahlfors--David regular sets and bilipschitz maps,
\textit{Ann.~Acad.~Sci. Fenn.~Math.}~\textbf{34} (2009), 487--502.

\bibitem{Mc} C.~McMullen, The Hausdorff dimension of general Sierpi\'nski carpets, \textit{Nagoya Math.~J.}~\textbf{96} (1984), 1--9.

\bibitem{Sh} E.~J.~McShane, Extension of range of functions,
\textit{Bull.~Amer.~Math.~Soc.}~\textbf{40} (1934), 837--842.

\bibitem{Mo} H.~Movahedi-Lankarani: An invariant of bi-Lipschitz maps, \textit{Fund.~Math.}~\textbf{143} (1993), 1--9.

\bibitem{SW} S.~Wang, personal communication.

\end{thebibliography}
\end{document}